\documentclass[11pt,a4paper]{article}

\usepackage[english]{babel}
\usepackage[utf8]{inputenc}
\usepackage[T1]{fontenc}
\usepackage{lmodern} \normalfont 
\DeclareFontShape{T1}{lmr}{bx}{sc} { <-> ssub * cmr/bx/sc }{}
\usepackage{parskip}
\usepackage{microtype}	

\usepackage{amssymb}
\usepackage[intlimits]{amsmath}
\usepackage{amsthm}
\usepackage{mathtools}
\usepackage{mathrsfs}
\usepackage{accents} 
\usepackage{etoolbox}
\usepackage{siunitx}
\usepackage{dsfont} 
\begingroup
    \makeatletter
    \@for\theoremstyle:=definition,remark,plain\do{%
        \expandafter\g@addto@macro\csname th@\theoremstyle\endcsname{%
            \addtolength\thm@preskip\parskip
            }%
        }
\endgroup

\usepackage{tikz}
\usetikzlibrary{calc,positioning,shapes}
\usetikzlibrary{patterns,decorations.pathmorphing,decorations.markings}
\usepackage{pgfplots}
\pgfplotsset{compat=newest}
\usepgfplotslibrary{external} 
\usepackage{standalone}
\usepackage{color}
\usepackage{graphicx}
\usepackage[margin=10pt,font=small,labelfont=bf,labelsep=endash]{caption}
\usepackage{subcaption}

\usepackage{algorithm}
\usepackage{algorithmicx}
\usepackage{algpseudocode}

\usepackage[left=3cm,right=3cm,top=3cm,bottom=3cm]{geometry}
\usepackage{booktabs}
\usepackage{paralist}

\usepackage{cite}
\usepackage{multirow} 
\usepackage{enumitem} 
\setlist[enumerate]{label=(\roman*)}
\let\oldbibliography\thebibliography
\renewcommand{\thebibliography}[1]{%
  \small
  \oldbibliography{#1}%
  \setlength{\itemsep}{0pt}%
}

\usepackage[colorlinks,citecolor=black,urlcolor=black]{hyperref}
\usepackage[nameinlink]{cleveref}

\numberwithin{equation}{section}
\numberwithin{table}{section}
\numberwithin{figure}{section}

\newtheorem{theorem}{Theorem}[section]

\newtheorem{hypothesis}[theorem]{Hypothesis} 

\newtheorem{lemma}[theorem]{Lemma}
\crefname{lemma}{lemma}{lemmata}
\Crefname{lemma}{Lemma}{Lemmata}
\newtheorem{corollary}[theorem]{Corollary}
\crefname{corollary}{corollary}{corollaries}
\Crefname{corollary}{Corollary}{Corollaries}

\crefname{assumption}{assumption}{assumptions}
\Crefname{assumption}{Assumption}{Assumptions}
\theoremstyle{definition}
\newtheorem{remark}[theorem]{Remark}
\AtEndEnvironment{remark}{\hfill\ensuremath{\diamondsuit}}
\theoremstyle{definition}
\newtheorem{defn}[theorem]{Definition}
\newtheorem{example}[theorem]{Example}
\AtEndEnvironment{example}{\hfill\ensuremath{\bigcirc}}

\newcommand{\R}{\ensuremath\mathbb{R}}
\newcommand{\mcal}[1]{\ensuremath\mathcal{#1}}
\newcommand{\calC}{\mcal{C}}
\newcommand{\calD}{\mcal{D}}
\newcommand{\wcalD}{\widetilde{\calD}}
\newcommand{\calK}{\mcal{K}}
\newcommand{\calM}{\mcal{M}}
\newcommand{\wcalM}{\widetilde{\calM}}
\newcommand{\calU}{\mcal{U}}

\newcommand{\ddt}{\ensuremath\frac{\mathrm{d}}{\mathrm{d}t}}

\DeclareMathOperator{\rank}{rank}

\newcommand{\state}{z}
\newcommand{\stateDim}{n}

\newenvironment{smallbmatrix}{\left[\begin{smallmatrix}}{\end{smallmatrix}\right]}

\newcommand{\timeInt}{\ensuremath\mathbb{I}}
\newcommand{\Za}{Z_{\mathrm{A}}}
\newcommand{\tZa}{\widetilde{Z}_{\mathrm{A}}}
\newcommand{\Zd}{Z_{\mathrm{D}}}
\newcommand{\tZd}{\widetilde{Z}_{\mathrm{D}}}
\newcommand{\Ta}{T_{\mathrm{A}}}
\newcommand{\tTa}{\widetilde{T}_{\mathrm{A}}}
\newcommand{\domain}[1]{\mathbb{D}_{#1}}
\newcommand{\WCFL}{S}
\newcommand{\WCFR}{T}
\newcommand{\ndif}{\stateDim_{\mathrm{d}}}
\newcommand{\nalg}{\stateDim_{\mathrm{a}}}

\newcommand{\delay}{\tau}
\newcommand{\shift}[2]{\sigma_{#1}\!\left(#2\right)}
\newcommand{\history}{\phi}
\newcommand{\mos}[2]{#1_{[#2]}}
\newcommand{\dmos}[2]{\dot{#1}_{[#2]}}
\newcommand{\maxMOS}{M}
\newcommand{\mosIndexSet}{\mathcal{I}}
\newcommand{\mosIndexSetEx}{\{1,\ldots,\maxMOS\}}

\newcommand{\dcheck}[1]{\check{\check{#1}}}
\newcommand{\dhat}[1]{\hat{\hat{#1}}}

\newcommand{\cZa}{\check{Z}_{\mathrm{A}}}
\newcommand{\cZd}{\check{Z}_{\mathrm{D}}}
\newcommand{\cTa}{\check{T}_{\mathrm{A}}}
\newcommand{\hZa}{\hat{Z}_{\mathrm{A}}}
\newcommand{\hZd}{\hat{Z}_{\mathrm{D}}}
\newcommand{\hTa}{\hat{T}_{\mathrm{A}}}

\title{Delay differential-algebraic equations in real-time dynamic substructuring\footnotemark[1]}
\author{Benjamin Unger\footnotemark[2]}
\date{\today}

\begin{document}

\maketitle
\renewcommand{\thefootnote}{\fnsymbol{footnote}}
\footnotetext[1]{
This work is funded by the DFG Collaborative Research Center 910 \emph{Control of self-organizing nonlinear systems: Theoretical methods and concepts of application}, project number 163436311.}
\footnotetext[2]{Institut f\"ur Mathematik,
Technische Universität Berlin, Str.\ des 17.~Juni~136,
10623~Berlin,
Germany,
\texttt{unger@math.tu-berlin.de}. }

%
%
\begin{abstract} Hybrid numerical-experimental testing is a standard approach for complex dynamical structures that are, on the one hand, not easy to model due to complexity and parameter uncertainty and, on the other hand, too expensive for full-scale experiments. The main idea is to subdivide the structure in a part that can be accurately simulated with numerical methods and an experimental component. The numerical simulation and the experiment are coupled in real-time by a so-called transfer system, which induces a time-delay into the system. In this paper, we study the solvability of the resulting hybrid numerical-experimental system, which is typically described by a set of nonlinear delay differential-algebraic equations, and extend existing results from the literature to this case.
 	
\vskip .3truecm

\vskip .3truecm

{\bf AMS(MOS) subject classification:} 34A09, 34A12, 34K32, 34K40
\end{abstract}


\section{Introduction}
We study existence and uniqueness of solutions for a specific class of nonlinear \emph{delay differential-algebraic equations} (DDAEs)
\begin{subequations}
\label{eq:DDAE:IVP}
\begin{align}
	\label{eq:DDAE}
	0 &= F(t,\state(t),\dot{\state}(t),\state(t-\delay)) & \text{for $t>0$}\\
	\intertext{with initial trajectory (also called \emph{history function})}
	\state(t) &= \history(t) &\text{for $t\in[-\delay,0]$}
\end{align}
\end{subequations}
that arise in the analysis of hybrid numerical-experimental testing, which is commonly used in earthquake engineering, see \cite{BurW08} and the references therein.
The main reason for a hybrid numerical-experimental setup is the fact that in some applications the description of the model in terms of differential equations is difficult due to its complex nature or uncertainty \cite{WilB01}. 
Since testing of a complete prototype may be prohibitively expensive, is it desirable to incorporate the benefits of actual testing with the benefits of numerical simulation. 
This is accomplished by subdividing the system under investigation into smaller subsystems, which are typically referred to as substructures; see \Cref{fig:subsystems} for an illustration. Alternatively, a bottom-up approach is facilitated by modern modeling languages such as \textsc{Modellica} (\url{https://www.modelica.org}) or \textsc{Matlab/Simulink} (\url{https://www.mathworks.org}). These frameworks compose the complete model by linking small components from a large library together. Such an automated modeling concept typically results in a combination of differential and algebraic equations, thus making the complete model a differential-algebraic equation (DAE).
\begin{figure}
	\centering
	\begin{tikzpicture}
		\tikzstyle{myEllipse} = [ellipse,outer sep=.2cm,inner sep=0.1cm, draw, thick, fill=black!25]

		\draw[fill=black!5,thick] (3.5,0) ellipse (6.5 and 2.2);
		
		\node (sys1) [myEllipse] at (0,0) {$
			\begin{aligned}
				0 &= \check{F}(t,\state_1,\dot{\state}_1,u_1),\\
				y_1 &= \check{G}(t,\state_1)
			\end{aligned}$};
		\node (sys2) [myEllipse] at (7,0) {$
			\begin{aligned}
				0 &= \hat{F}(t,\state_2,\dot{\state}_2,u_2),\\
				y_2 &= \hat{G}(t,\state_2)
			\end{aligned}$};
		
		\node[color=black!90] at (3.5,1.4) {\textbf{\textsc{Physical System}}};
		
		\draw[<-, thick,transform canvas={yshift=-.3em}] (sys1) -- (sys2.west) node [below,pos=0.5] {$u_1 = y_2$};
		\draw[->, thick,transform canvas={yshift=.3em}] (sys1) -- (sys2.west) node [above,pos=0.5] {$u_2 = y_1$}; ;
	\end{tikzpicture}
	\caption{Decomposition of a physical system into substructures}
	\label{fig:subsystems}
\end{figure}
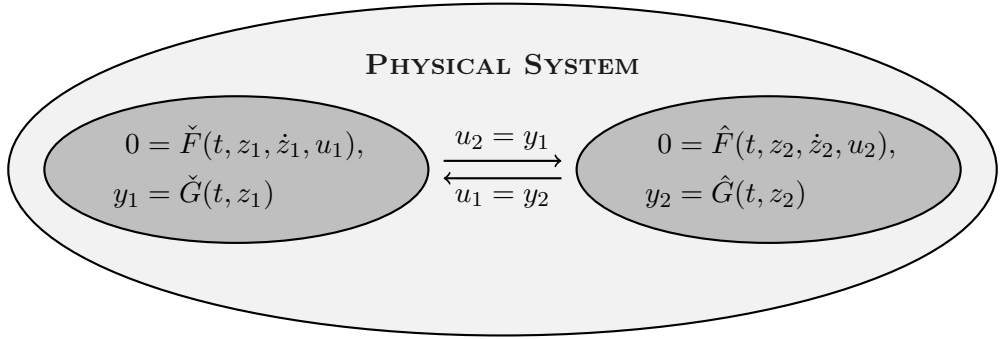
 
Having decomposed the system into smaller substructures, the  numerical-experimental paradigm is to test only a specific substructure experimentally, while the remainder of the system is simulated numerically. To ensure dynamical interaction in real-time, the experiment and the numerically simulation have to happen simultaneously with a possibility to interact through a well-defined interface. In \emph{real-time dynamic substructuring} or \emph{hardware-in-the-loop testing} \cite{BurW08} the interface, called \emph{transfer system}, is typically provided by a set of hydraulic actuators and a set of sensors. Since the dynamic characteristic of any actuator includes a response delay \cite{HorIKN99, WalSNWK05}, the resulting system is a DDAE. Note that further delays might be present, which arise, for instance, from data acquisition, computation, or digital signal processing. In many applications, these delays are small compared to the actuator delay and may thus be neglected in the modeling process; for more details, we refer to  \cite{KryBGHW06} and the references therein. The model equations for the hybrid numerical-experimental approach are discussed in \Cref{sec:modelEquations}, yielding the coupled DDAE~\eqref{eq:hybridSystem}. The approach is exemplified by a coupled pendulum-mass-spring-damper system (cf.~\Cref{fig:substructuringPMSD}) taken from \cite{KryBGHW06}. Our two main results are:
\begin{enumerate}
	\item \Cref{lem:hybrid:linear:shifted:regular} details that if the subsystems are linear time-invariant regular DAEs, then the system arising from the hybrid approach is regular in the sense that the associated DAE (see~\Cref{def:associatedDAE}) is regular. The result is extended to nonlinear strangeness-free subsystems in \Cref{thm:regularStrangenessFree} and extended to higher index in \Cref{thm:hybrid:shifted:regular}. In particular, we give a bound on the index of the hybrid system in terms of the index of the subsystems. Notably, in order to obtain the regularity results we have to shift certain equations of the hybrid model in time.
	\item The regularity result combined with the method of steps \cite{BelZ03} allows us to establish existence results for the initial trajectory problem \eqref{eq:DDAE:IVP}, provided that the DDAE~\eqref{eq:DDAE} is not advanced. The details are presented in \Cref{def:RNA} and \Cref{thm:solvabilityDDAE}.
\end{enumerate}

Our analysis extends existing solvability results to a more general class of nonlinear DDAEs. Classical solution theory for DDAEs that need not be shifted is developed in \cite{AscP95,Cor00} for nonlinear DDAEs with a special structure and for linear time-invariant DDAEs in \cite{Cam80}. Shifting and its consequences are studied in \cite{Cam95,HaM12,HaMS14,HaM16,Ha15,TreU19,AhrU19ppt,Pop06,Ha18}. 
Numerical time integration methods are developed and analyzed in \cite{GugH08,GugH14,HuCH20,LinTB18,Hau97,BakPT02,ShaG06,HaM16,BetCT15,TiaYK14,CheZ12,HuCH18}. Most of the references for the numerical methods require that there is no need to shift equations and that the associated DAE has differentiation index one. Notable exceptions are provided for instance in \cite{AscP95,Hau97,BakPT02,HaM16,AhrU19ppt}. 
The stability and asymptotic stability of certain classes of DDAEs is studied in \cite{LuzR06,Log98,Fri02,Mic11,DuLMT13,CamL09,ZhuP97,ZhuP98,XuVSL02,Ha18a,LinT15,MazBNC20}. 

After this introduction we discuss a motivational example in \Cref{subsec:motivationalExample} and give a general description of the hybrid numerical-experimental system in \Cref{subsec:modelEquations}. As a preparation for our main results, we review the strangeness index \cite{KunM06} and the differentiation index \cite{BreCP96} for differential-algebraic equations in \Cref{sec:DAEs} with a particular focus on coupled systems. In particular, we show in \Cref{ex:subsystems:Index:1,ex:coupledSfreeSystems} that the index of a coupled system is not necessarily related to the index of the subsystems.

\paragraph*{Notation}
As is common in the literature for delay equations, $\ddt$ denotes the derivative from the right and we write $\dot{f} = \ddt f$, $\ddot{f} = \left(\ddt\right)^2 f$, and $f^{(k)} \vcentcolon= (\ddt)^k f$. Similarly, the space $\calC^k(\timeInt,\R^\stateDim)$ denotes the space of $k$-times continuously differentiable functions from the time interval $\timeInt$ into $\R^{\stateDim}$. The ring of polynomials with real coefficients is denoted by $\R[s]$. For the concatenation of vectors $v_1,\ldots,v_k\in\R^{\stateDim}$ we use the notation $[v_1;\ldots;v_k] \vcentcolon= \begin{bmatrix}
	v_1^T & \ldots v_k^T
\end{bmatrix}^T$, where $v_i^T\in\R^{1\times\stateDim}$ denotes the transpose of $v_i$.

\section{Motivational example and model equations}
\label{sec:modelEquations}

\subsection{Coupled Pendulum-Mass-Spring-Damper system}
\label{subsec:motivationalExample}
We start our exposition on hybrid numerical-experimental systems with a motivational example taken from \cite{KryBGHW06} consisting of a pendulum that is coupled to a mass-spring-damper system. For our example, we consider the mass-spring-damper system as the numerical simulation and the pendulum as the experiment; see \Cref{fig:substructuringPMSD} for an illustration.
\begin{figure}
	\centering
	\begin{subfigure}[b]{0.4\textwidth}
		\centering
		\begin{tikzpicture}
			\tikzstyle{spring}=[thick,decorate,decoration={zigzag,pre length=0.3cm,post length=0.3cm,segment length=6}]
			\tikzstyle{damper}=[thick,decoration={markings,  
			  mark connection node=dmp,
			  mark=at position 0.5 with 
			  {
			    \node (dmp) [thick,inner sep=0pt,transform shape,rotate=-90,minimum width=15pt,minimum height=3pt,draw=none] {};
			    \draw [thick] ($(dmp.north east)+(2pt,0)$) -- (dmp.south east) -- (dmp.south west) -- ($(dmp.north west)+(2pt,0)$);
			    \draw [thick] ($(dmp.north)+(0,-5pt)$) -- ($(dmp.north)+(0,5pt)$);
			  }
			}, decorate]
			
			\node (Mass) [fill,black!30,minimum width=3cm,minimum height=1.5cm] at (0,0) {};
			\node [left,xshift=-.2cm] at (Mass) {$M$};
			\node [above right] at (Mass) {$(x_1,y_1)$};
			\node (ground) at (Mass.south) [yshift=-2.3cm,fill,pattern=north east lines,draw=none,minimum width=5cm,minimum height=.3cm,anchor=north] {};
			\draw [very thick] (ground.north west) -- (ground.north east);
			\draw [spring] ($(ground.north) - (1,0)$) -- ($(Mass.south)-(1,0)$) node [midway, left,xshift=-.2cm] {$K$};
			\draw [damper] ($(ground.north) + (1,0)$) -- ($(Mass.south)+(1,0)$) node [midway, left,xshift=-.4cm] {$C$};
			
			\draw[gray,->] ($(Mass) + (0,0)$) -- ($(Mass) + (2.5,0)$) node[black,right] {$x$};
			\draw[gray,->] ($(Mass) + (0,0)$) -- ($(Mass) + (0,1.5)$) node[black,right] {$y$};
			
			\node (pendulum) at ($(Mass) + (2.5,-1.5)$) {};
			
			\draw[gray,dashed,thick,->] (pendulum) arc (-31:-6:2.9cm);
			\draw[gray,dashed,thick,->] (pendulum) arc (-31:-56:2.9cm);
			
			\draw [fill=black] (pendulum) circle (0.15);
			\node [right,xshift=.2cm] at (pendulum) {$m$};
			\node [below,yshift=-.15cm] at (pendulum) {$(x_2,y_2)$};
			\draw [very thick] (0,0) -- (pendulum) node [pos=0.7,above right] {$L$};			
			
			\draw[fill=black] (Mass) circle (0.1);
		\end{tikzpicture}
		\caption{Fully coupled system}
		\label{fig:substructuringPMSD:coupled}
	\end{subfigure}
	\begin{subfigure}[b]{0.58\textwidth}
		\centering
		\begin{tikzpicture}
			\tikzstyle{spring}=[thick,decorate,decoration={zigzag,pre length=0.3cm,post length=0.3cm,segment length=6}]
			\tikzstyle{damper}=[thick,decoration={markings,  
			  mark connection node=dmp,
			  mark=at position 0.5 with 
			  {
			    \node (dmp) [thick,inner sep=0pt,transform shape,rotate=-90,minimum width=15pt,minimum height=3pt,draw=none] {};
			    \draw [thick] ($(dmp.north east)+(2pt,0)$) -- (dmp.south east) -- (dmp.south west) -- ($(dmp.north west)+(2pt,0)$);
			    \draw [thick] ($(dmp.north)+(0,-5pt)$) -- ($(dmp.north)+(0,5pt)$);
			  }
			}, decorate]
			
			\node (Mass) [fill,black!30,minimum width=2cm,minimum height=1cm] at (0,0) {};
			\node (Fext) at ($(Mass) + (0,1.2)$) {};
			\draw [very thick,->] ($(Mass)$) -- (Fext) node[below right] {$F_{\mathrm{ext}}$};
			\node [left,xshift=-.2cm] at (Mass) {$M$};
			\node (ground) at (Mass.south) [yshift=-1.5cm,fill,pattern=north east lines,draw=none,minimum width=2.5cm,minimum height=.3cm,anchor=north] {};
			\draw [very thick] (ground.north west) -- (ground.north east);
			\draw [spring] ($(ground.north) - (.4,0)$) -- ($(Mass.south)-(.4,0)$) node [midway, left,xshift=-.15cm] {$K$};
			\draw [damper] ($(ground.north) + (.4,0)$) -- ($(Mass.south)+(.4,0)$) node [midway, right,xshift=.3cm] {$C$};
			
			\draw[fill=black] (Mass) circle (0.1);
			
			\node (numMod) [align=center,anchor=north] at ($(ground.south)-(0,.2)$) {\textsc{numerical}\\\textsc{model}};
			
			\draw[rounded corners] ($(ground.west |- Fext) + (-.2,.2)$) rectangle ($(ground.east |- numMod.south) + (.2,-.2)$);
			
			\begin{scope}[xshift=5cm]
				\node (Mass2) {};
				\node (Fmeasured) at ($(Mass2) + (0,1.2)$) {};
				\draw [very thick,->] ($(Mass2)$) -- (Fmeasured) node[below right] {$F_{\mathrm{pendulum}}$};
				\draw[fill=black] (Mass2) circle (0.1);
				\node (ground2) at ($(Mass2)-(0,0.5)$) [xshift=.7cm,yshift=-1.5cm,fill,pattern=north east lines,draw=none,minimum width=2.5cm,minimum height=.3cm,anchor=north] {};
				\draw [very thick] (ground2.north west) -- (ground2.north east);
				
				\node (pendulum) at ($(Mass2) + (1.1,-1.5)$) {};
			
				\draw[gray,dashed,thick,->] (pendulum) arc (-53.7:-33.7:1.86cm);
				\draw[gray,dashed,thick,->] (pendulum) arc (-53.7:-73.7:1.86cm);
			
				\draw [fill=black] (pendulum) circle (0.15);
				\node [right,xshift=.2cm] at (pendulum) {$m$};
				\draw [very thick] (0,0) -- (pendulum) node [pos=0.7,above right] {$L$};			
			
				\node (actuator) at ($(Mass2) - (0,0.9)$) {};
				\node (actBottom) at (actuator |- ground2.north) {};
				\draw [fill=black!50] ($(actBottom) - (0.1,0)$) rectangle ($(actuator) + (0.1,0)$);
				\draw ($(actuator) - (0.05,0)$) -- ($(actuator) - (0.05,-.55)$) -- ($(actuator) - (0.4,-.55)$) -- ($(actuator |- Mass2) - (0.4,0.1)$) -- ($(actuator |- Mass2) + (0.4,-0.1)$) -- ($(actuator) + (0.4,.55)$) -- ($(actuator) + (0.05,.55)$) -- ($(actuator) + (0.05,0)$);
				
				\node (expAct) [align=center,anchor=north] at ($(ground2.south)-(0,.2)$) {\textsc{experiment}\\\textsc{+ actuator}};
			
				\draw[rounded corners] ($(ground2.west |- Fmeasured) + (-.2,.2)$) rectangle ($(ground2.east |- expAct.south) + (.2,-.2)$);
			\end{scope}
			
			\draw [thick,->,shorten >=0.5cm,shorten <=0.5cm,] (Mass -| ground.east) -- (Mass2 -| ground2.west) node[midway,align=center] {\small adjust\\[.2em]\small position};
			
			\draw [thick,->,shorten >=0.5cm,shorten <=0.5cm,] (ground2.west) -- (ground.east) node[midway,align=center] {\small send\\[.2em]\small $F_{\mathrm{pendulum}}$};
		\end{tikzpicture}
		\caption{Hybrid numerical-experimental setup}
		\label{fig:substructuringPMSD:hybrid}
	\end{subfigure}
	\caption{Real-time dynamic substructuring for a coupled pendulum-mass-spring-damper system}
	\label{fig:substructuringPMSD}
\end{figure}
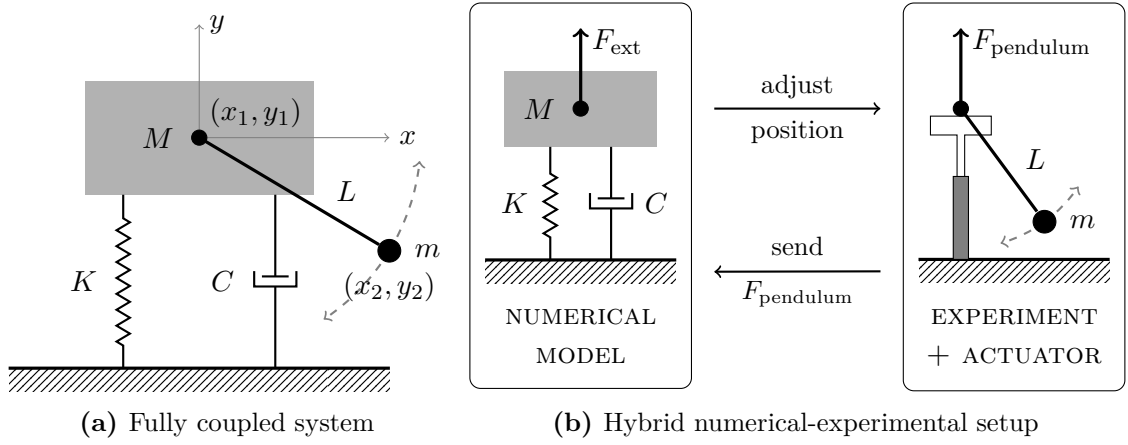
For our numerical model, we assume that the mass $M$ is mounted on a linear spring and a linear viscous damper. The resulting equation of motion for the mass-spring-damper system is given by
\begin{equation}
	\label{eq:MSDsystem}
	M\ddot{y}_1 + C\dot{y}_1+ Ky_1 = F_{\mathrm{ext}},
\end{equation}
where $C$ and $K$ denote the damping and the stiffness coefficient, respectively. The external force, which in this scenario will be provided by the pendulum is given by $F_{\mathrm{ext}}$. We assume that the pendulum is given by a point mass $m$ that is attached to the spring-mass-damper system via a massless rod of length $L$. Assuming no friction, the model for the pendulum is given by
\begin{equation}
	\label{eq:pendulum}
	\begin{aligned}
		m\ddot{x}_2 &= -2\lambda x_2,\\
		m\ddot{y}_2 &= -2\lambda (y_2-y_1) - m\mathsf{g},\\
		0 &= x_2^2 + (y_2-y_1)^2 - L^2,
	\end{aligned}
\end{equation}
with gravitational constant $\mathsf{g}$ and Langrange multiplier $\lambda$. By Newton's second law, the force that the pendulum generates in $y$-direction is given by 
\begin{displaymath}
	F_{\mathrm{pendulum}} = -2\lambda (y_2-y_1) - m\mathsf{g}.
\end{displaymath}
Consequently, the equations of motion for the fully coupled system (as depicted in \Cref{fig:substructuringPMSD:coupled}) are given by
\begin{equation}
	\label{eq:PMSD:coupled}
	\begin{aligned}
			M\ddot{y}_1 + C\dot{y}_1 + Ky_1 &= -2\lambda (y_2-y_1) - m\mathsf{g},\\
			m\ddot{x}_2 &= -2\lambda x_2,\\
			m\ddot{y}_2 &= -2\lambda (y_2-y_1) - m\mathsf{g},\\
			0 &= x_2^2 + (y_2-y_1)^2 - L^2,
	\end{aligned}
\end{equation}
with unknown functions $y_1,x_2,y_2$, and $\lambda$. In the hybrid numerical-experimental setup (cf.~\Cref{fig:substructuringPMSD:hybrid}), the actuator introduces a time-delay $\delay>0$ into the system, which is assumed constant. The delay can be understood as an offset in time between the mass-spring-damper dynamics \eqref{eq:MSDsystem} and the pendulum dynamics \eqref{eq:pendulum}. In particular, we have to replace $t$ by $t-\delay$ in the pendulum dynamics \eqref{eq:pendulum} and the force $F_{\mathrm{pendulum}}$. Thus, the complete mathematical description for the hybrid numerical-experimental setup is given by
\begin{subequations}
	\label{eq:PMSD:hybrid}
	\begin{align}
		\label{eq:PMSD:hybrid:1}M\ddot{y}_1 + C\dot{y}_1+ Ky_1 &= -2\lambda(\cdot-\delay) (y_2(\cdot-\delay)-y_1(\cdot-\delay)) - m\mathsf{g},\\
		\label{eq:PMSD:hybrid:2}m\ddot{x}_2(\cdot-\delay) &= -2\lambda(\cdot-\delay) x_2(\cdot-\delay),\\
		\label{eq:PMSD:hybrid:3}m\ddot{y}_2(\cdot-\delay) &= -2\lambda(\cdot-\delay) (y_2(\cdot-\delay)-y_1(\cdot-\delay)) - m\mathsf{g},\\
		\label{eq:PMSD:hybrid:4}0 &= x_2(\cdot-\delay)^2 + (y_2(\cdot-\delay)-y_1(\cdot-\delay))^2 - L^2.
	\end{align}
\end{subequations}
If we introduce new variables for $\dot{y}_1$, $\dot{x}_2$, and $\dot{y}_2$, then we can rewrite~\eqref{eq:PMSD:hybrid} in the form~\eqref{eq:DDAE}. 

\subsection{Hybrid numerical-experimental system}
\label{subsec:modelEquations}
For the general description of the model equations, we assume that we have already subdivided the complete model into two sub-models, which later on represent the numerical part and the experimental part. For an illustrate we refer to \Cref{fig:subsystems}. The first subsystem is described by the descriptor system
\begin{subequations}
	\label{eq:DAE:numerical}
	\begin{align}
		0 &= \check{F}(t,\state_1,\dot{\state}_1,u_1),\\
		y_1 &= \check{G}(t,\state_1)		
	\end{align}
\end{subequations}
with state $\state_1(t)\in\R^{\stateDim_1}$, input $u_1(t)\in\R^{m}$, and output $y_1(t)\in\R^{p}$. The second subsystem is given by
\begin{subequations}
	\label{eq:DAE:experimental}
	\begin{align}
		0 &= \hat{F}(t,\state_2,\dot{\state}_2,u_2),\\
		y_2 &= \hat{G}(t,\state_2)		
	\end{align}
\end{subequations}
with $\state_2(t)\in\R^{\stateDim_2}$, $u_2(t)\in\R^{m_2}$, and $y_2(t)\in\R^{p_2}$. The complete model is given by imposing the interconnections
\begin{equation}
	\label{eq:interconnection}
	u_1(t) = y_2(t)\qquad\text{and}\qquad u_2(t) = y_1(t).
\end{equation}
In particular, we assume $m_1 = p_2$ and $m_2 = p_1$. The complete model as depicted in \Cref{fig:subsystems} is thus given by the implicit equation
\begin{equation}
	\label{eq:completeModel}
	0 = \begin{bmatrix}
		\check{F}(t,\state_1,\dot{\state}_1,\hat{G}(t,\state_2))\\
		\hat{F}(t,\state_2,\dot{\state}_2,\check{G}(t,\state_1))
	\end{bmatrix}
\end{equation}
with initial conditions
\begin{equation}
	\label{eq:initialCondition}
	\state_1(0) = \zeta_1\qquad\text{and}\qquad \state_2(0) = \zeta_2.
\end{equation}
Let us emphasize that the initial values $\zeta_1\in\R^{\stateDim_1}$ and $\zeta_2\in\R^{\stateDim_2}$ have to satisfy some consistency condition for a (classical) solution to exist. For further details we refer to \cite{KunM06} and the forthcoming \Cref{sec:DAEs}.

\begin{example}
	\label{ex:PMSD:firstOrder}
	To recast the coupled pendulum-mass-spring-damper system, we first have to transform the systems to first order. Introducing new variables for the velocities and renaming, we obtain
	\begin{align*}
		\check{F}(t,\state_1,\dot{\state}_1,u_1) &= \begin{bmatrix}
			\dot{\state}_{1,1} - \state_{1,2}\\
			M\dot{\state}_{1,2} + C\state_{1,2} + K\state_{1,1} - u_1
		\end{bmatrix}, & \check{G}(t,\state_1) &= \state_{1,1},\\
		\hat{F}(t,\state_2,\dot{\state}_2,u_2) &= \begin{bmatrix}
			\dot{\state}_{2,1} - \state_{2,4}\\
			\dot{\state}_{2,2} - \state_{2,5}\\
			\state_{2,6} - u_2\\
			m\dot{\state}_{2,4} + 2\state_{2,3}\state_{2,1}\\
			m\dot{\state}_{2,5} + 2\state_{2,3}(\state_{2,2} - u_2) + m\mathsf{g}\\
			\state_{2,1}^2 + (\state_{2,2} - u_2)^2 - L^2
		\end{bmatrix}, & 
		\hat{G}(t,\state_2) &= -2\state_{2,3}(\state_{2,2} - \state_{2,6}) - m\mathsf{g}.
	\end{align*}
	Note that we have introduced the artificial variable $\state_{2,6}$ to account for the feedthrough.
\end{example}

If both subsystems are linear time-invariant, then we write
\begin{equation}
	\label{eq:DAE:linearSubsystems}
	\begin{aligned}
	\check{F}(t,\state_1,\dot{\state}_1,u_1) &= \check{E}\dot{\state}_1 - \check{A}\state_1 - \check{B}u_1 + \check{f}(t), \quad& 
	\check{G}(t,\state_1) &= \check{C}\state_1,\\
	\hat{F}(t,\state_2,\dot{\state}_2,u_2) &= \hat{E}\dot{\state}_2 - \hat{A}\state_2 - \hat{B}u_2 + \hat{f}(t), & \hat{G}(t,\state_2) &= \hat{C}\state_2,
	\end{aligned}
\end{equation}
with external forcing functions $\check{f}$ and $\hat{f}$. 
such that the complete model~\eqref{eq:completeModel} is given by
\begin{equation}
	\label{eq:completeModel:linear}
	\begin{bmatrix}
		\check{E} & \\
		& \hat{E}
	\end{bmatrix}\begin{bmatrix}
		\dot{\state}_1\\
		\dot{\state}_2
	\end{bmatrix} = \begin{bmatrix}
		\check{A} & \check{B}\hat{C}\\
		\hat{B}\check{C} & \hat{A}
	\end{bmatrix}\begin{bmatrix}
		\state_1\\
		\state_2
	\end{bmatrix} + \begin{bmatrix}
		\check{f}\\
		\hat{f}
	\end{bmatrix}.
\end{equation}

Our standing assumption is that the first model is simulated numerically, while the second model is tested experimentally. Following the discussion above, the transfer system that realizes the numerical results in real-time within the experiment is delayed, such that the second model technically acts at a different time point. The hybrid numerical-experimental model, which we study in this paper, is thus given by
\begin{equation}
	\label{eq:hybridSystem}
	0 = \begin{bmatrix}
		\check{F}(t,\state_1(t),\dot{\state}_1(t),\hat{G}(t-\delay,\state_2(t-\delay)))\\
		\hat{F}(t-\delay,\state_2(t-\delay),\dot{\state}_2(t-\delay),\check{G}(t-\delay,\state_1(t-\delay)))
	\end{bmatrix},
\end{equation}
which in the linear case simplifies to
\begin{equation}
	\label{eq:hybridSystem:linear}
	\begin{split}
	\begin{bmatrix}
		\check{E} & \\
		& 0
	\end{bmatrix}\begin{bmatrix}
		\dot{\state}_1(t)\\
		\dot{\state}_2(t)
	\end{bmatrix} + \begin{bmatrix}
		0 & \\
		& \hat{E}
	\end{bmatrix}\begin{bmatrix}
		\dot{\state}_1(t-\delay)\\
		\dot{\state}_2(t-\delay)
	\end{bmatrix} = \hspace{16em}\\
	\begin{bmatrix}
		\check{A} & \\
		 & 0
	\end{bmatrix}\begin{bmatrix}
		\state_1(t)\\
		\state_2(t)
	\end{bmatrix} + \begin{bmatrix}
		 & \check{B}\hat{C}\\
		\hat{B}\check{C} & \hat{A}
	\end{bmatrix}\begin{bmatrix}
		\state_1(t-\delay)\\
		\state_2(t-\delay)
	\end{bmatrix} + \begin{bmatrix}
		\check{f}\\
		\hat{f}(t-\delay)
	\end{bmatrix}.
	\end{split}
\end{equation}
Note that if the hybrid model is initialized at time $t_0$, then the numerical simulation starts at $t_0$, while the experimental part starts at $t_0+\delay$. In particular, it is sufficient to prescribe an initial trajectory solely for the experimental part, i.e., only for $\state_2$.

\begin{remark}
	In general, the input and output dimensions $m_i$ and $p_i$ may not be compatible and instead, a relation of the form
	\begin{equation}
		\label{eq:couplingRelation}
		0 = \calK(t,u_1,y_1,u_2,y_2)
	\end{equation}
	can be imposed to connect the two subsystems. For the hybrid numerical-experimental model \eqref{eq:hybridSystem} the coupling relation \eqref{eq:couplingRelation} has to be replaced with
	\begin{displaymath}
		0 = \calK(t,u_1(t),y_1(t-\delay),u_2(t-\delay),y_2(t-\delay)).
	\end{displaymath}
	To simplify notation we restrict ourselves to the setting described above and emphasize that all results can be generalized to the implicit coupling relation \eqref{eq:couplingRelation}.
\end{remark}

\section{Differential-algebraic equations}
\label{sec:DAEs}
The two subsystems~\eqref{eq:DAE:experimental} and~\eqref{eq:DAE:numerical} are nonlinear DAEs of the form
\begin{equation}
	\label{eq:DAE}
	F(t,\state(t),\dot{\state}(t),u(t)) = 0,
\end{equation}
where $\state(t)\in\R^{\stateDim}$ and $u(t)\in\R^{m}$ denote, respectively, the \emph{state} and \emph{control} of the system, which is posed on the (compact) time interval $\timeInt \vcentcolon= [0,T]$. The function 
\begin{displaymath}
	F\colon \timeInt\times \domain{\state} \times \domain{\dot{\state}} \times \domain{u}\to\mathbb{R}^n
\end{displaymath}
with open sets $\domain{\state},\domain{\dot{\state}}\subseteq \R^{\stateDim}$, $\domain{u}\subset\R^{m}$ is assumed to be sufficiently smooth. A function $\state\in\calC^1(\timeInt,\R^{\stateDim})$ is called a \emph{(classical) solution} of \eqref{eq:DAE} if $\state$ satisfies \eqref{eq:DAE} pointwise. An initial condition $\state(t_0) = \state_0\in\R^{\stateDim}$ is called \emph{consistent} if for a given control $u$, the associated initial value problem (IVP)
\begin{equation}
	\label{eq:IVP}
	F(t,\state(t),\dot{\state}(t),u(t)) = 0,\qquad \state(t_0) = \state_0
\end{equation}
has at least one solution. Throughout this paper we assume that \eqref{eq:IVP} is \emph{regular}, i.\,e., for every sufficiently smooth input $u$ the DAE~\eqref{eq:DAE} is solvable and the solution is unique for every consistent initial condition.

The control problem \eqref{eq:DAE} is often studied in the behavior framework \cite{PolW98}, see for instance \cite{KunM01,CamKM12}. Hereby, a new variable $\xi = [\state, u]$ is introduced that includes the state and control variable such that the problem is reduced to the analysis of an underdetermined DAE \cite{KunM01}, i.\,e., the meaning of the variables is not distinguished any more. One big advantage of this formalism is that the analysis determines the free variables in the system, which might not be the original control variables, and hence need to be reinterpreted.
Since our main goal is to study the IVP~\eqref{eq:IVP} with a prescribed input function $u$ this viewpoint is not possible. For given $u$ we can study the restricted problem
\begin{equation}
	\label{eq:rDAE}
	\widetilde{F}(t,\state(t),\dot{\state}(t)) = 0,\qquad \state(t_0) = \state_0,
\end{equation}
with $\widetilde{F}(t,\state,\dot{\state}) = F(t,\state,\dot{\state},u)$. 

If the partial derivative $\tfrac{\partial}{\partial \dot{\state}} \widetilde{F}$ is singular, then the solution $\state$ of \eqref{eq:rDAE} may depend on derivatives of $\widetilde{F}$. The difficulties arising with these differentations are classified by so called \emph{index} concepts (cf.~\cite{Meh15} for a survey). In this paper, we make use of the \emph{strangeness index} concept \cite{KunM06}, which is -- roughly speaking -- a generalization of the \emph{differentiation index} \cite{BreCP96} to under- and overdetermined systems. The advantage of the strangeness index is that it preserves the algebraic constraints in the system, which in turn prevents numerical methods to drift away from the solution manifold \cite{KunM04}. The strangeness index is based on the \emph{derivative array} \cite{Cam87a} of level $\ell$, defined as
\begin{equation}
	\label{eq:derivativeArray}
	\wcalD_\ell\left(t,\state,\eta\right) \vcentcolon= \begin{bmatrix}
		\widetilde{F}(t,\state,\dot{\state})\\
		\ddt\widetilde{F}(t,\state,\dot{\state})\\
		\vdots\\
		\left(\ddt\right)^{\!\ell} \widetilde{F}(t,\state,\dot{\state})
	\end{bmatrix}\in\R^{(\ell+1)\stateDim}\qquad\text{with}\ \eta\vcentcolon=[\dot{\state};\ldots;\state^{(\mu+1)}].
\end{equation}
\begin{hypothesis}
	\label{hyp:1}
	There exist integers $\mu$ and $a$ such that the set
	\begin{displaymath}
		\wcalM_\mu \vcentcolon= \left\{\left(t,\state,\eta\right)\in\R^{(\mu+2)\stateDim+1}\ \big|\ \wcalD_{\mu}\left(t,\state,\eta\right) = 0\right\}
	\end{displaymath}
	associated with $\widetilde{F}$ is nonempty and such that for every $(t_0,\state_0,\eta_0)\in\wcalM_{\mu}$, there exists a (sufficiently small) neighborhood $\widetilde{\calU}$ in which the following properties hold:
	\begin{enumerate}
		\item We have $\rank(\tfrac{\partial}{\partial \eta} \wcalD_{\mu}) = (\mu+1)\stateDim-a$ on $\wcalM_{\mu}\cap\widetilde{\calU}$ such that there exists a smooth matrix function $\tZa$ of size $(\mu+1)\stateDim\times a$ and pointwise maximal rank that satisfies $\tZa^T \tfrac{\partial}{\partial \eta} \wcalD_{\mu}=0$.
		\item We have $\rank(\tZa^T \tfrac{\partial}{\partial \state}\wcalD_{\mu}) = a$ on $\wcalM_{\mu}\cap\widetilde{\calU}$ such that there exists a smooth matrix function $\tTa$ of size $\stateDim\times d$ with $d \vcentcolon= \stateDim-a$ and pointwise maximal rank, satisfying $\tZa^T \left(\tfrac{\partial}{\partial \state} \wcalD_{\mu}\right) \tTa = 0$.
		\item We have $\rank(\tfrac{\partial \widetilde{F}}{\partial\dot{\state}} \tTa) = d$ on $\wcalM_{\mu}\cap\widetilde{\calU}$ such that there exists a smooth matrix function $\tZd$ of size $\stateDim\times d$ and pointwise maximal rank, satisfying $\rank(\tZd^T \tfrac{\partial\widetilde{F}}{\partial\dot{\state}} \tTa) = d$.
	\end{enumerate}
\end{hypothesis}

The smallest possible $\mu$ for which \Cref{hyp:1} is satisfied is called \emph{strangeness index} of the DAE~\eqref{eq:rDAE}. If \Cref{hyp:1} is satisfied with $\mu=0$, then the DAE~\eqref{eq:rDAE} is called \emph{strangeness-free}. The quantities $a$ and $d$ are, respectively, the numbers of differential and algebraic equations contained in the DAE~\eqref{eq:rDAE}. Using the matrix functions $\tZd$ and $\tZa$, the DAE~\eqref{eq:rDAE} can (locally) be reformulated as
\begin{subequations}
	\label{eq:sFree}
	\begin{align}
		\label{eq:sFreeDifferential} 0 &= \widetilde{D}(t,\state,\dot{\state}) \vcentcolon= \left(\tZd^T \widetilde{F}\right)(t,\state,\dot{\state}),\\
		\label{eq:sFreeAlgebraic} 0 &= \widetilde{A}(t,\state) \vcentcolon= \left(\tZa^T\wcalD_{\mu}\right)(t,\state),
	\end{align}
\end{subequations}
which itself is strangeness-free and every solution of \eqref{eq:rDAE} also solves \eqref{eq:sFree}. Hereby we call \eqref{eq:sFreeDifferential} the \emph{differential part} of \eqref{eq:rDAE} and \eqref{eq:sFreeAlgebraic} the \emph{algebraic part}. Note that although $\tZa$ and $\wcalD_{\mu}$ may depend on derivatives of $\state$ it can be shown (cf.~\cite{KunM98}) that their product only depends on $t$ and $\state$. Unfortunately, a solution of \eqref{eq:sFree} is not necessarily a solution of \eqref{eq:rDAE}. However, if we assume in addition, that \Cref{hyp:1} is satisfied with characteristic values $\mu,a,d$ and $\mu+1,a,d$, then for every initial value $\state_{\mu+1,0}\in\calM_{\mu+1}$ there exists a unique solution of \eqref{eq:sFree} and this solution (locally) solves \eqref{eq:rDAE} (see \cite[Theorem~4.13]{KunM06}). As a direct consequence, an initial value $\state_0$ is consistent if and only if it is contained in the \emph{consistency set}
\begin{equation}
	\label{eq:DAE:consistencySet}
	(t_0,\state_0) \in \widetilde{\mathbb{M}} \vcentcolon= \left\{(t,\state)\in\R^{\stateDim+1} \mid \widetilde{A}(t,\state) = 0\right\}.
\end{equation}
If state transformations are allowed, then the implicit function theorem allows us to (locally) rewrite the strangeness-free DAE~\eqref{eq:sFree} as
\begin{equation}
	\label{eq:sFreeExplicit}
	\dot{\xi} = \widetilde{\mathcal{L}}(t,\xi), \qquad
	\zeta = \widetilde{\mathcal{R}}(t,\xi)
\end{equation}
with $\xi(t)\in\R^d$ and $\zeta(t)\in\R^a$. For the detailed derivation we refer to \cite[Cha.~4.1]{KunM06}. Let $\state = \mathcal{T}(t,\xi,\zeta)$ denote the transformation for the state. Then, the ordinary differential equation (ODE)
\begin{equation}
	\label{eq:uODE}
	\dot{\state} = \widetilde{\mathfrak{f}}(t,\state) \vcentcolon= \mathcal{T}\left(t, \widetilde{\mathcal{L}}(t,\xi),
		\left(\tfrac{\partial}{\partial \xi} \widetilde{\mathcal{R}}\right)(t,\xi)\widetilde{\mathcal{L}}(t,\xi) + \left(\tfrac{\partial}{\partial t}\widetilde{\mathcal{R}}\right)(t,\xi)\right),
\end{equation}
is called the \emph{underlying ODE} for the DAE \eqref{eq:rDAE} and is the basis of the \emph{differentiation index} \cite{BreCP96}, which is defined as $\mu+1$ if $\frac{\partial}{\partial\dot{\state}}\widetilde{F}$ is singular and $0$ otherwise \cite[Cor.~3.46]{KunM06}. 

If we want to solve the DAE \eqref{eq:DAE} numerically, we are not only interested in the existence of solutions but also that the solution of the initial value problem~\eqref{eq:rDAE} is unique and depends continuously on the data. For DAEs, the so-called \emph{well-posedness} can be be formulated as follows \cite[Theorem~4.12]{KunM06}.

\begin{theorem}
	\label{thm:wellPosednessDAE}
	Let $\widetilde{F}$ as in \eqref{eq:rDAE} be sufficiently smooth and satisfy \Cref{hyp:1}. Let $\state^\star\in\mathcal{C}^1(\timeInt,\mathbb{R}^{\stateDim})$ be a sufficiently smooth solution of \eqref{eq:DAE}. Let the (nonlinear) operator $\widetilde{\mathcal{F}}\colon \mathbb{D}\to\mathbb{Y}$, $\mathbb{D}\subseteq\mathbb{Z}$ open, be defined by
	\begin{equation}
		\label{eq:wellPosednessOperator}
		\widetilde{\mathcal{F}}(\state)(t) = \begin{bmatrix}
			\dot{\xi} - \widetilde{\mathcal{L}}(t,\xi(t))\\
			\zeta - \widetilde{\mathcal{R}}(t,\xi(t))
		\end{bmatrix},
	\end{equation}
	with the Banach spaces
	\begin{displaymath}
		\mathbb{Z} \vcentcolon= \{z\in \calC(\timeInt,\R^{\stateDim} \mid \xi\in\calC^1(\timeInt,\R^{d}), \xi(t_0) = 0\},\qquad \mathbb{Y} \vcentcolon= \calC(\timeInt,\R^{\stateDim})
	\end{displaymath}
	according to \eqref{eq:sFreeExplicit}. Then $\state^\star$ is a regular solution of the strangeness-free problem
	\begin{displaymath}
		\widetilde{\mathcal{F}}(\state) = 0
	\end{displaymath}
	in the following sense. There exists a neighborhood $\calU_\state\subseteq \mathbb{Z}$ of $\state^\star$ and a neighborhood $\mathcal{V}\subseteq\mathbb{Y}$ of the origin such that for every $f\in\mathcal{V}$ the equation
	\begin{displaymath}
		\widetilde{\mathcal{F}}(\state) = f
	\end{displaymath}
	has a unique solution $\state\in\calU_\state$ that depends continuously on $f$. In particular, $\state^\star$ is the unique solution in $\calU_\state$ belonging to $f = 0$.
\end{theorem}

In order to apply the theory to the original equation~\eqref{eq:DAE} we have to ensure that the characteristic values $\mu$, $a$, and $d$ do not depend on the chosen input $u$. A simple way to guarantee this, is to ensure that the rank assumptions in \Cref{hyp:1} hold for all sufficiently smooth input functions. The derivative array~\eqref{eq:derivativeArray} with explicit dependency on $u$ takes the form
\begin{equation*}
	\label{eq:derivativeArray2}
	\calD_\ell\left(t,\state,\eta,u,\dot{u},\ldots,u^{(\ell)}\right) \vcentcolon= \begin{bmatrix}
		F(t,\state,\dot{\state},u)\\
		\ddt F(t,\state,\dot{\state},u)\\
		\vdots\\
		\left(\ddt\right)^{\!\ell} F(t,\state,\dot{\state},u)
	\end{bmatrix}\in\R^{(\ell+1)\stateDim}\qquad\text{with}\ \eta\vcentcolon=[\dot{\state};\ldots;\state^{(\mu+1)}].
\end{equation*}

\begin{hypothesis}
	\label{hyp:2}
	There exist integers $\mu$ and $a$, and matrix functions $\Za(\cdot)\in \R^{(\mu+1)\stateDim\times a}$, $\Ta(\cdot)\in\R^{\stateDim\times d}$, and $\Zd(\cdot)\in\R^{\stateDim\times d}$ with pointwise maximal rank and $d\vcentcolon= n-a$ such that for every sufficiently smooth $u$ the set
	\begin{displaymath}
		\calM_\mu \vcentcolon= \left\{\left(t,\state,\eta,u,\ldots,u^{(\mu)}\right)\in\R^{(\mu+2)\stateDim+(\mu+1)m+1}\ \big|\ \calD_{\mu}\left(t,\state,\eta,u,\ldots,u^{(\mu)}\right) = 0\right\}
	\end{displaymath}
	associated with $F$ is nonempty and such that for every $(t_0,\state_0,\eta_0,u_0,\ldots,u_0^{(\mu)})\in\calM_{\mu}$, there exists a (sufficiently small) neighborhood $\calU$ in which the following properties hold:
	\begin{enumerate}
		\item We have $\rank(\tfrac{\partial}{\partial \eta} \calD_{\mu}) = (\mu+1)\stateDim-a$ and $\Za^T \tfrac{\partial}{\partial \eta} \calD_{\mu}=0$ on $\calM_{\mu}\cap\calU$.
		\item We have $\rank(\Za^T \tfrac{\partial}{\partial \state}\calD_{\mu}) = a$ and $\Za^T \left(\tfrac{\partial}{\partial \state} \calD_{\mu}\right) \Ta = 0$ on $\calM_{\mu}\cap\calU$.
		\item We have $\rank(\tfrac{\partial F}{\partial\dot{\state}} \Ta) = d$ and $\rank(\Zd^T \tfrac{\partial F}{\partial\dot{\state}} \Ta) = d$ on $\calM_{\mu}\cap\calU$.
	\end{enumerate}
\end{hypothesis}

\begin{remark}
	Note that similarly as in \Cref{hyp:1} the existence of the matrix functions $\Za$, $\Ta$, and $\Zd$ in \Cref{hyp:2} follows from the constant rank assumptions and a smooth version of the singular value decomposition as in~\cite[Thm.~3.9 and Thm.~4.3]{KunM06}.
\end{remark}

\begin{example}
	It is easy to see that $\check{F}$ in \Cref{ex:PMSD:firstOrder} satisfies \Cref{hyp:2} with $\mu=0$. The equation for the pendulum, summarized in $\hat{F}$, are in Hessenberg-form and therefore satisfy \Cref{hyp:2} with strangeness index $\mu=2$ \cite[Thm.~4.23]{KunM06}.
\end{example}

Following the analysis in \cite{KunM98} that leads to the strangeness-free formulation~\eqref{eq:sFree} we observe that the functions $D$ and $A$ may depend on $u$ and its derivatives. Due to the local character of \Cref{hyp:2} we can  assume that $D$ does not depend on derivatives of $u$. In any case, \Cref{hyp:2} yields the (local) reformulation
\begin{subequations}
	\label{eq:sFree:control}
	\begin{align}
		\label{eq:sFreeDifferential:control} 0 &= D(t,\state,\dot{\state},u) \vcentcolon= \left(\Zd^T F\right)(t,\state,\dot{\state},u),\\
		\label{eq:sFreeAlgebraic:control} 0 &= A\left(t,\state,u,\dot{u},\ldots,u^{(\mu)}\right) \vcentcolon= \left(\Za^T\calD_{\mu}\right)\left(t,\state,u,\dot{u},\ldots,u^{(\mu)}\right),
	\end{align}
\end{subequations}
which itself is strangeness-free. The corresponding explicit form~\eqref{eq:sFreeExplicit} and the underlying ODE~\eqref{eq:uODE} therefore take the form
\begin{equation}
	\label{eq:sFreeExplicit:control}
	\dot{\xi} = \mathcal{L}(t,\xi,u), \qquad
	\zeta = \mathcal{R}(t,\xi,u,\dot{u},\ldots,u^{(\mu)})
\end{equation}
and
\begin{equation}
	\label{eq:uODE:control}
	\dot{\state} = \mathfrak{f}\left(t,\state,u,\ldots,u^{(\mu+1)}\right).
\end{equation}
Clearly, if a system satisfies \Cref{hyp:2}, then it also satisfies \Cref{hyp:1} (with given~$u$) and thus all previous results hold as well.
\begin{remark}
	Although derivatives of $u$ up to order $\mu$, respectively $\mu+1$ appear in the algebraic equation~\eqref{eq:sFreeAlgebraic:control}, the explicit algebraic equation~\eqref{eq:sFreeExplicit:control}, and the underlying ODE~\eqref{eq:uODE:control}, respectively, it is not necessary, that they actually depend on it. 
\end{remark}
In the linear case, i.e., when \eqref{eq:DAE} takes the form
\begin{equation}
	\label{eq:DAE:linear}
	E\dot{\state} = A\state + Bu + f
\end{equation}
with matrices $E,A\in\R^{\stateDim\times\stateDim}$, $B\in\R^{\stateDim\times m}$ and external forcing $f$, the analysis reduces to the analysis of the matrix pencil $(E,A)$. We say the $(E,A)$ is \emph{regular}, if $\det(sE-A)\in\R[s]\setminus\{0\}$ is not the zero polynomial. Otherwise, $(E,A)$ is called \emph{singular}. In this case, the Weierstra\ss{} canonical form \cite{Gan59b} separates the algebraic and differential part.
\begin{theorem}[Weierstra\ss{} canonical form]
	\label{thm:WCF}
	The matrix pencil $(E,A)\in \left(\R^{\stateDim\times\stateDim}\right)^2$ is regular if and only if there exist matrices nonsingular matrices $\WCFL,\WCFR\in\R^{\stateDim\times\stateDim}$ such that
	\begin{equation}
		\label{eq:WCFDAE}
		\WCFL E\WCFR= \begin{bmatrix}
			I_{\ndif} & 0\\
			0 & N
		\end{bmatrix}\qquad\text{and}\qquad 
		\WCFL A\WCFR = \begin{bmatrix}
			J & 0 \\
			0 & I_{\nalg}\end{bmatrix},
	\end{equation}
	where the matrices $J\in\R^{\ndif\times\ndif}$ and $N\in\R^{\nalg\times\nalg}$ are in Jordan canonical form and $N$ is nilpotent. 
\end{theorem}

One can show that the index of nilpotency $\nu\in\mathbb{N}$ of $n$, i.e., the unique number $\nu$ that satisfies $N^\nu = 0$ and $N^{\nu-1}\neq 0$, does not depend on the particular choice of $\WCFL$ and $\WCFR$ and thus is called the \emph{index}.
The Weierstra\ss{} canonical form separates the DAE~\eqref{eq:DAE:linear} into the differential equation
\begin{displaymath}
	\dot{\state}_{\mathrm{d}} = J\state_{\mathrm{d}} + B_{\mathrm{d}}u + f_{\mathrm{d}},
\end{displaymath}
which for continuous $u$ and $f_{\mathrm{d}}$ is uniquely solvable for any initial condition $\state_{\mathrm{d}}(0) = \state_{\mathrm{d},0}$, and the algebraic equation
\begin{equation}
	\label{eq:DAE:linear:algebraicEquation}
	N\dot{\state}_{\mathrm{a}} = \state_{\mathrm{a}} +  B_{\mathrm{a}}u + f_{\mathrm{a}}.
\end{equation}
Hereby, we use $
	\begin{smallbmatrix}
		\state_{\mathrm{d}}\\
		\state_{\mathrm{a}}
	\end{smallbmatrix} \vcentcolon= \WCFR^{-1}\state$, $
	\begin{smallbmatrix}
		B_{\mathrm{d}}\\
		B_{\mathrm{a}}
	\end{smallbmatrix} \vcentcolon= \WCFL B$, and $
	\begin{smallbmatrix}
		f_{\mathrm{d}}\\
		f_{\mathrm{a}}
	\end{smallbmatrix} \vcentcolon= \WCFL f$. It is easy to see (cf.~\cite[Lem.~2.8]{KunM06}) that the unique solution of \eqref{eq:DAE:linear:algebraicEquation} is given by
\begin{equation}
	\label{eq:DAE:linear:algebraic:Sol}
	\state_{\mathrm{a}} = -\sum_{j=0}^{\nu-1} N^i\left(B_{\mathrm{a}}u^{(j)} + f_{\mathrm{a}}^{(j)}\right),
\end{equation}
provided that $u$ and $f_{\mathrm{a}}$ are sufficiently smooth.  
In particular, an initial condition $\state_{\mathrm{a}}(0) = \state_{\mathrm{a},0}$ is consistent if and only if it satisfies~\eqref{eq:DAE:linear:algebraic:Sol}, i.e., in the linear case we consistency set~\eqref{eq:DAE:consistencySet} is given by
\begin{displaymath}
	\mathbb{M} \vcentcolon= \left\{(t,\state)\in\R^{\stateDim+1}\,\left|\, \state_{\mathrm{a}}(t) = -\sum_{j=0}^{\nu-1} N^i\left(B_{\mathrm{a}}u^{(j)}(t) + f^{(j)}(t)\right)\right.\right\}.
\end{displaymath}
We observe that the explicit solution formula~\eqref{eq:DAE:linear:algebraic:Sol} is obtained by differentiating~\eqref{eq:DAE:linear:algebraicEquation} $\nu-1$ times. Another differentiation therefore allows us to rewrite the DAE~\eqref{eq:DAE:linear} as an ODE, showing that the index (of nilpotency) and the differentiation index coincide. In particular, we conclude \cite[Cor.~3.46]{KunM06}
\begin{displaymath}
	\nu = \begin{cases}
		0, & \text{if } a = 0,\\
		\mu + 1, & \text{if } a > 0,
	\end{cases}
\end{displaymath}
where $a$ is the characteristic quantity defined in \Cref{hyp:2}.

\begin{remark}
	The transformation of a matrix pencil with nonsingular matrices as in \Cref{thm:WCF} defines an equivalence relation \cite[Lem.~2.2]{KunM06}. We write $(E,A) \sim (\tilde{E},\tilde{A})$ if there exists nonsingular matrices $\WCFL,\WCFR$ with
	$\tilde{E} = \WCFL E\WCFR$ and $\tilde{A} = \WCFL A\WCFR$.
\end{remark}
Before we continue our discussion let us emphasize that in general, there is no relation between the regularity of the subsystems~\eqref{eq:DAE:numerical} and~\eqref{eq:DAE:experimental} and the regularity of the coupled system~\eqref{eq:completeModel}. Also the index from the subsystems might differ from the index of the coupled system. As an immediate consequence, the splitting of the system into smaller subsystems is a delicate task that has to be performed carefully.
\begin{example}
	\label{ex:subsystems:Index:1}
	Consider the linear DAE
	\begin{equation}
		\label{eq:exSub:totalSys}
		\begin{bmatrix}
			1 & 0 & 0\\
			0 & 0 & 0\\
			0 & 0 & 0
		\end{bmatrix}\begin{bmatrix}
			\dot{x}_1\\\dot{x}_2\\\dot{x}_3
		\end{bmatrix} = \begin{bmatrix}
			0 & c & 0\\
			c & 0 & 1\\
			0 & 1 & -1
		\end{bmatrix}\begin{bmatrix}
			x_1\\x_2\\x_3
		\end{bmatrix} + \begin{bmatrix}
			f_1\\f_2\\f_3
		\end{bmatrix}
	\end{equation}
	with external forcing function $f = [f_1, f_2, f_3]$ and parameter $c\in\R$. It is easy to see that for any $c\in\R$ the system has differentiation index $\nu = 1$. Splitting the system into $z_1 = [x_1, x_2]$ and $z_2 = x_3$ we obtain the two subsystems
	\begin{subequations}
		\begin{align}
			\label{eq:exSub:sub:1}\begin{bmatrix}
				1 & 0\\
				0 & 0
			\end{bmatrix}\begin{bmatrix}
				\dot{x}_1\\
				\dot{x}_2
			\end{bmatrix} &= \begin{bmatrix}
				0 & c\\
				c & 0
			\end{bmatrix}\begin{bmatrix}
				x_1\\
				x_2
			\end{bmatrix} + \begin{bmatrix}
				0\\
				1
			\end{bmatrix}u_1 + \begin{bmatrix}
				f_1\\f_2
			\end{bmatrix}\\
			\label{eq:exSub:sub:2}0 &= -x_3 + u_2 + f_3.
		\end{align}
	\end{subequations}
	The second subsystem~\eqref{eq:exSub:sub:2} has differentiation index $\nu = 1$. For the first subsystem~\eqref{eq:exSub:sub:1} we observe that for $c=0$ the pencil of the DAE is singular. For $c\neq 0$ the pencil is regular with index $\nu = 2$, which is bigger than the index of the coupled system.
\end{example}

\begin{example}
	\label{ex:coupledSfreeSystems}
	For $i=1,2$ we consider the subsystems
	\begin{displaymath}
		\begin{bmatrix}
			1 & 0\\
			0 & 0
		\end{bmatrix}\dot{\state}_i = \begin{bmatrix}
			a_i & 0\\
			0 & 1
		\end{bmatrix}\state_i + \begin{bmatrix}
			b_{i,1} & b_{i,2}\\
			c_{i,1} & c_{i,2}
		\end{bmatrix}u_i,
	\end{displaymath}
	which are already in Weierstra\ss{} canonical form \eqref{eq:WCFDAE} with index $\nu=1$. The coupled system with coupling relations $u_1 = \state_2$ and $u_2 = \state_1$ is given by the linear DAE $E\dot{\state} = A\state$ with
	\begin{displaymath}
		E \vcentcolon=
		\begin{bmatrix}
			1 & 0 & 0 & 0\\
			0 & 0 & 0 & 0\\
			0 & 0 & 1 & 0\\
			0 & 0 & 0 & 0
		\end{bmatrix}, \qquad A \vcentcolon= \begin{bmatrix}
			a_1 & 0 & b_{1,1} & b_{1,2}\\
			0 & 1 & c_{1,1} & c_{1,2}\\
			b_{2,1} & b_{2,2} & a_2 & 0\\
			c_{2,1} & c_{2,2} & 0 & 1
		\end{bmatrix},\qquad \text{and}\qquad
		\state \vcentcolon= \begin{bmatrix}
			\dot{\state}_1\\\dot{\state}_2
		\end{bmatrix}.
	\end{displaymath}
	From 
	\begin{displaymath}
		(E,A) \sim \left(\begin{bmatrix}
			1 & 0 & 0 & 0\\
			0 & 1 & 0 & 0\\
			0 & 0 & 0 & 0\\
			0 & 0 & 0 & 0
		\end{bmatrix}, \begin{bmatrix}
			a_1 & b_{1,1} & 0 & b_{1,2}\\
			b_{2,1} & a_2 & b_{2,2} & 0\\
			0 & c_{1,1} & 1 & c_{1,2}\\
			c_{2,1} & 0 & c_{2,2} & 1
		\end{bmatrix}\right)
	\end{displaymath}
	we immediately observe that $(E,A)$ has differentiation index $\nu=1$ if and only if $c_{1,2}c_{2,2} \neq 1$. Otherwise, we obtain
	\begin{displaymath}
		(E,A) \sim \left(\begin{bmatrix}
			1 & 0 & 0 & 0\\
			0 & 1 & 0 & 0\\
			0 & 0 & 0 & 0\\
			0 & 0 & 0 & 0
		\end{bmatrix}, \begin{bmatrix}
			a_1 & b_{1,1} & b_{1,2} & 0\\
			b_{2,1} & a_2-c_{1,1}b_{2,1} & - b_{21}c_{12} & 0\\
			c_{2,1} & -c_{1,1}c_{2,2} & 0 & 0\\
			0 & 0 & 0 & 1
		\end{bmatrix}\right)
	\end{displaymath}
	showing that also $\nu=2$, $\nu=3$, and $(E,A)$ singular are possible.
\end{example}

\begin{remark}
	\label{rem:pH}
	If both subsystems are port-Hamiltonian systems \cite{BeaMXZ18}, then, under reasonable conditions, the coupled system itself is again a port-Hamiltonian system. In this case, \cite[Thm.~4.3]{MehMW18} implies that the differentiation index of the coupled system is at most $\nu=2$.
\end{remark}

\section{The method of steps}
\label{sec:MOS}

The standard procedure to solve initial trajectory problems for delay equations is via successive integration on the time intervals $[(i-1)\delay,i\delay)$ with $i=1,\ldots,\maxMOS$, where $\maxMOS\in\mathbb{N}$ is the smallest integer such that $T \leq \maxMOS\delay$. This approach is commonly referred to as the \emph{method of steps} \cite{HaM16}, see also \cite{BelZ03,Cam80}. For the DDAE~\eqref{eq:DDAE} we therefore introduce for $i\in\mosIndexSet \vcentcolon= \mosIndexSetEx$
\begin{equation}
	\label{eq:restriction}
	\begin{aligned}
	\mos{\state}{i} \colon [0,\delay]&\to\R^{\stateDim},\qquad & t &\mapsto \state(t+(i-1)\delay),\\
	\mos{F}{i} \colon [0,\delay] \times \domain{\state} \times \domain{\dot{\state}} \times \domain{\shift{\delay}\state}  &\to\R^{\stateDim}, & (t,\state,y,z) &\mapsto F(t+(i-1)\delay,\state,y,x),\\
	\mos{\state}{0} \colon [0,\delay]&\to\R^{\stateDim}, & t &\mapsto \history(t-\delay).
	\end{aligned}
\end{equation}
Then we have to solve for each $i\in\mosIndexSetEx$ the DAE 
\begin{subequations}
	\label{eq:sequenceDAEs}
	\begin{align}
		\label{eq:methodOfSteps}	
		0 &= \mos{F}{i}(t,\mos{\state}{i},\dmos{\state}{i},\mos{\state}{i-1}), & t\in[0,\delay),\\
		\label{eq:initialConditionDAE}\mos{\state}{i}(0) &= \mos{\state}{i-1}(\delay^-),
	\end{align}
\end{subequations}
with right continuation
\begin{equation}
	\label{eq:rightContinuation}
	\mos{\state}{i-1}(\delay^-) \vcentcolon= \lim_{t\nearrow \delay} \mos{\state}{i-1}(t).
\end{equation}
If \eqref{eq:sequenceDAEs} is uniquely solvable (provided that the initial value $\mos{\state}{i-1}(\delay^-)$ is consistent), then we can construct the solution of \eqref{eq:IVP} on the successive time intervals $[(i-1)\delay,i\delay)$. In general, we cannot expect a smooth transition of the solution between these intervals. This is, for instance, illustrated with several examples in \cite{Ung18}. We therefore use the following solution concept.
\begin{defn}[Solution concept]
	\label{def:solutionConcept}
	Assume that $F$ in the DDAE~\eqref{eq:DDAE} and the history function $\history$ are sufficiently smooth. We call $\state\in\mathcal{C}(\timeInt,\R^{\stateDim})$ a \emph{solution} of \eqref{eq:DDAE:IVP} if for all $i\in\mosIndexSet$ the restriction $\mos{\state}{i}$ of $\state$ as in \eqref{eq:restriction} is a solution of \eqref{eq:sequenceDAEs}. We call the history function $\history\colon[-\delay,0]\to\R^{\stateDim}$ \emph{consistent} if the initial value problem \eqref{eq:DDAE:IVP} has at least one solution.
\end{defn}
We emphasize that in order to check if a history function is consistent, we actually have to compute a solution of the initial value problem~\eqref{eq:DDAE:IVP}. This is in contrast to the DAE theory, where it suffices to compute the consistency set~\eqref{eq:DAE:consistencySet}. To account for this issue, we adopt the following definition from \cite{Ung18}, which ensures that we can at least ensure a solution in the interval $[0,\tau)$.

\begin{defn}[Admissible history function]
	The history function $\history$ is called \emph{admissible} for the DDAE~\eqref{eq:DDAE} if the initial condition
	\begin{displaymath}
		\mos{\state}{1}(0) = \history(0)
	\end{displaymath}
	is consistent for the DAE~\eqref{eq:sequenceDAEs} with $i=1$.
\end{defn}

Following the discussion from \Cref{sec:DAEs}, consistent initial values are characterized by the consistency set~\eqref{eq:DAE:consistencySet}. We therefore have to assume that the DAE
\begin{equation}
	0 = \mos{F}{1}(t,\mos{\state}{1},\dmos{\state}{1},\history(t-\delay))
\end{equation}
satisfies~\Cref{hyp:1}. In order to simplify the discussion, we make the following definition.

\begin{defn}
	\label{def:associatedDAE}
	The DAE that is obtained from the DDAE~\eqref{eq:DDAE} by substituting $\state(t-\delay)$ with a control function $u(t)$ is called the \emph{associated DAE} for the DDAE~\eqref{eq:DDAE}. We say that the DDAE~\eqref{eq:DDAE} satisfies \Cref{hyp:2} if its associated DAE satisfies \Cref{hyp:2}.
\end{defn}

Suppose now that the DDAE~\eqref{eq:DDAE} satisfies \Cref{hyp:2} with strangeness index $\mu$. Following the discussion in \Cref{sec:DAEs} we observe that the algebraic equation (which now takes the form of a difference equation) is given by
\begin{equation}
	\label{eq:differenceEquation}
	0 = A\left(t,\state(t),\state(t-\delay),\dot{\state}(t-\delay),\ldots,\state^{(\mu)}(t-\delay\right).
\end{equation}
Since the set of consistent initial values is described by~\eqref{eq:differenceEquation}, we immediately obtain the following result.
\begin{lemma}
	\label{lem:admissibleHistory}
	Assume that the history function $\history$ is sufficiently smooth and the DDAE~\eqref{eq:DDAE} satisfies \Cref{hyp:2} with strangeness index $\mu$. Then~$\history$ is admissible for the DDAE~\eqref{eq:DDAE} if and only if
	\begin{equation}
		\label{eq:admissibleHistory}
		0 = A\left(t,\history(0),\history(-\delay),\dot{\history}(-\delay),\ldots,\history^{(\mu)}(-\delay)\right).
	\end{equation}
\end{lemma}

\Cref{lem:admissibleHistory} requires that the DDAE satisfies \Cref{hyp:2}, which in turn implies that the associated DAE is regular. Unfortunately, this is only a sufficient condition for the existence of a unique solution for the initial trajectory problem \eqref{eq:DDAE:IVP}, see for instance \cite{TreU19,HaM16}. It is easy to see that the associated DAE for the hybrid numerical-experimental model~\eqref{eq:hybridSystem} is not regular and therefore does not satisfy \Cref{hyp:2} and hence \Cref{lem:admissibleHistory} does not apply to \eqref{eq:hybridSystem}.

One strategy to resolve this issue is to find a reformulation of the DDAE~\eqref{eq:DDAE} by shifting certain equations. This is achieved either by a combined shift-and-derivative array and the so-called \emph{shift index} \cite{HaM16,Ha15}, or by some kind of compress-and-shift algorithm \cite{Cam95,HaM12,HaMS14,TreU19}. The idea of the latter algorithm is to identify (after a potential transformation of the equations -- the compression step) which equations do not depend on the current state but solely on the past state. These equations are then shifted in time and the procedure is iterated. The special structure of the hybrid numerical-experimental model thus immediately suggest to shift the second block row of equations, yielding
\begin{equation}
	\label{eq:hybridSystem:shifted}
	0 = F(t,\state,\dot{\state},\shift{\delay}{\state}) \vcentcolon= \begin{bmatrix}
		\check{F}(t,\state_1,\dot{\state}_1,\shift{\delay}{\hat{G}(t,\state_2))}\\
		\hat{F}(t,\state_2,\dot{\state}_2,\check{G}(t,\state_1))
	\end{bmatrix},
\end{equation}
with $\state(t) \vcentcolon= [\state_1(t);\state_2(t)]\in\R^{\stateDim}$, $\stateDim\vcentcolon= \stateDim_1 + \stateDim_2$ and shift operator
\begin{displaymath}
	\shift{\delay}{f}(t) \vcentcolon= f(t-\delay).
\end{displaymath}
In the linear case \eqref{eq:hybridSystem:shifted} simplifies to
\begin{equation}
	\label{eq:hybridSystem:linear:shifted}
	\begin{bmatrix}
		\check{E} & \\
		& \hat{E}
	\end{bmatrix}\begin{bmatrix}
		\dot{\state}_1\\
		\dot{\state}_2
	\end{bmatrix} = \begin{bmatrix}
		\check{A} & 0\\
		\hat{B}\check{C} & \hat{A}
	\end{bmatrix}\begin{bmatrix}
		\state_1\\
		\state_2
	\end{bmatrix} + \begin{bmatrix}
		0 & \check{B}\hat{C}\\
		0 & 0
	\end{bmatrix}\begin{bmatrix}
		\shift{\delay}{\state_1}\\
		\shift{\delay}{\state_2}
	\end{bmatrix} + 
	\begin{bmatrix}
		\check{f}\\
		\hat{f}
	\end{bmatrix}.
\end{equation}
We immediately obtain
\begin{displaymath}
	\det\left(\begin{bmatrix}
		s\check{E}-\check{A} & 0\\
		-\hat{B}\check{C} & s\hat{E}-\hat{A}
	\end{bmatrix}\right) = \det(s\check{E}-\check{A})\det(s\hat{E}-\hat{A})
\end{displaymath}
and thus have proven the next result.
\begin{lemma}
	\label{lem:hybrid:linear:shifted:regular}
	The matrix pencil of the associated DAE for the linear shifted hybrid numerical-experimental system~\eqref{eq:hybridSystem:linear:shifted} is regular if and only if the associated DAEs of the linear subsystems~\eqref{eq:DAE:linearSubsystems} are regular.
\end{lemma}
\begin{remark}
 	In the terminology of \cite{HaM16}, the hybrid numerical-experimental system~\eqref{eq:hybridSystem:linear} has \emph{shift index} $\kappa=1$. In the literature, shifting of equations, i.e., systems with shift index $\kappa>0$, are often referred to as noncausal and hence not physical. The hybrid numerical-experimental setup details that shifting of equations can also occur if the dynamics of the subsystems affect the overall dynamic at different time instants. 
\end{remark}
Before we proceed let us emphasize that shifting of equations potentially enlarges the solution space of the initial trajectory problem for the differential equation. 
\begin{example}
	\label{ex:shiftingSolutionSpace}
	Consider the DDAE
	\begin{subequations}
	\label{eq:shift}
	\begin{align}
		\label{eq:shift:1} \dot{x}_1(t) &= x_2(t-\delay) + f(t),\\
		\label{eq:shift:2} 0 &= x_2(t-\delay) - g(t).
	\end{align}
	\end{subequations}
	Notice that the second equation constitutes a restriction for the initial trajectory. Indeed, if we prescribe the initial trajectory
	\begin{equation}
		\label{eq:initialTrajectory}
			x_1(t) = \history_1(t), \qquad 
			x_2(t) = \history_2(t),
		\qquad\text{for $t\leq 0$},
	\end{equation}
	then a solution cannot exist if $\history_2(t) \neq g(t+\delay)$ for $t\in[-\delay,0]$. If $\history_2(t) = g(t+\delay)$ for $t\in[-\delay,0]$, then the solution of the initial trajectory problem \eqref{eq:shift},\eqref{eq:initialTrajectory} is given by
	\begin{displaymath}
		x_1(t) = \history_1(0) + \int_0^t g(s) + f(s)\,\mathrm{d}s,\qquad x_2(t) = g(t+\delay)\qquad\text{for $t\geq 0$}.
	\end{displaymath}
	In particular, the solution space for $x_1$ is parameterized by $\history_1(0)$ and thus a one-dimensional vector space. 
	If we however replace \eqref{eq:shift:2} with the shifted equation
	\begin{equation}
		\label{eq:shift:2:modified}
		x_2(t) = g(t+\delay)
	\end{equation}
	and consider the initial trajectory problem \eqref{eq:shift:1},\eqref{eq:shift:2:modified},\eqref{eq:initialTrajectory}, then for any history function $\history$ that satisfies $\history_2(0) = g(\delay)$ the solution of \eqref{eq:shift:1},\eqref{eq:shift:2:modified},\eqref{eq:initialTrajectory} for $t\in[0,\delay]$ is given by
	\begin{displaymath}
		x_1(t) = \history_1(0) + \int_0^t \history_2(s-\tau) + f(s)\,\mathrm{d}s,\qquad x_2(t) = g(t+\tau),
	\end{displaymath}
	such that the solution space for $x_1$ is infinite-dimensional.
\end{example}
\begin{remark}
	The shifted hybrid system~\eqref{eq:hybridSystem:shifted} showcases, that only an initial trajectory for the experimental system $\hat{F}$ is required, which is in agreement with the discussion after \eqref{eq:hybridSystem:linear}. This is no contradiction to \Cref{ex:shiftingSolutionSpace} since the numerical and experimental part are initialized at different time points. 
\end{remark}

If the linear subsystems are regular, then \Cref{lem:hybrid:linear:shifted:regular} together with \cite[Thm.~4]{TreU19} immediately implies existence and uniqueness of solutions of the initial trajectory problem for the DDAE~\eqref{eq:hybridSystem:linear:shifted} in the space of piecewise-smooth distributions \cite{Tre09-thesis}. Unfortunately, it is not immediately clear, what the index of the matrix pencil of the associated DAE is.
\begin{example}
	\label{ex:index:shiftedHybridSys}
	Consider the matrix pencil
	\begin{displaymath}
		\left(\left[\begin{array}{cc|cc}
			0 & 1 & 0 & 0\\
			0 & 0 & 0 & 0\\\hline
			0 & 0 & 0 & 1\\
			0 & 0 & 0 & 0
		\end{array}\right], \left[\begin{array}{cc|cc}
			1 & 0 & 0 & 0\\
			0 & 1 & 0 & 0\\\hline
			a & b & 1 & 0\\
			c & d & 0 & 1
		\end{array}\right]\right) \sim \left(\left[\begin{array}{cccc}
			0 & 1 & 0 & 0\\
			0 & 0 & 0 & 0\\
			0 & 0 & 0 & 1\\
			0 & 0 & 0 & 0
		\end{array}\right], \left[\begin{array}{cccc}
			1 & 0 & 0 & 0\\
			0 & 1 & 0 & 0\\
			0 & 0 & 1 & 0\\
			c & 0 & 0 & 1
		\end{array}\right]\right)
	\end{displaymath}
	of the hybrid numerical-experimental system~\eqref{eq:hybridSystem:linear:shifted}, where both subsystems have differentiation index $\nu=1$. If $c=0$, then the pencil also has index $\nu=1$, otherwise the index is $\nu=2$.
\end{example}
In particular, the index of the shifted hybrid numerical-experimental model depends on the coupling functions $\check{G}$ and $\hat{G}$. As a direct consequence, \Cref{hyp:2} has to be checked for each example separately, since it is not clear a-priori, what the corresponding strangeness index $\mu$ is. A notable exception is provided in the case that both subsystems are strangeness-free.
\begin{theorem}
	\label{thm:regularStrangenessFree}
	Suppose that the subsystems \eqref{eq:DAE:numerical} and \eqref{eq:DAE:experimental} are strangeness-free, i.e., satisfy \Cref{hyp:2} with characteristic values $\check{\mu}=\hat{\mu} = 0$, $\check{a},\hat{a}, \check{d}$, and $\hat{d}$, respectively. If~$\delay>0$, then the shifted hybrid numerical-experimental model \eqref{eq:hybridSystem:shifted} satisfies \Cref{hyp:2} with characteristic values $\mu=0$, $a=\check{a}+\hat{a}$, and $d=\check{d}+\hat{d}$. 
\end{theorem}

\begin{proof}
	Let $\cZa$, $\cTa$, $\cZd$, and $\hZa$, $\hTa$, $\hZd$ denote the matrix functions from \Cref{hyp:2} for the subsystems \eqref{eq:DAE:numerical} and \eqref{eq:DAE:experimental}, respectively. Define $a \vcentcolon= \check{a} + \hat{a}$ and accordingly, $d = n-a = n_1 - \check{a} + n_2 - \hat{a} = \check{d} + \hat{d}$. Choose $\hTa^\star$ such that $\begin{bmatrix}
			\hTa & \hTa^\star
		\end{bmatrix}$ is nonsingular. From \Cref{hyp:2} we deduce that
	\begin{displaymath}
		\left(\hZa^T \frac{\partial \hat{F}}{\partial \state_2}\hTa^\star\right)(t,\state_2,\dot{\state}_2,\hat{G}(t,\state_1))
	\end{displaymath}
	is nonsingular.  Define (omitting arguments) the matrix functions 
	\begin{align*}
		\Za &\vcentcolon= \begin{bmatrix}
		\cZa & 0\\
		0 & \hZa\end{bmatrix}, & 
		\Ta &\vcentcolon= \begin{bmatrix}
			\cTa & 0\\
			-\hTa^\star\left(\hZa^T \frac{\partial \hat{F}}{\partial \state_2}\hTa^\star \right)^{-1} \cZa^T\frac{\partial \hat{F}}{\partial u_2}\frac{\partial \check{G}}{\partial \state_1}\hTa & \hTa
		\end{bmatrix}, &
		\Zd &\vcentcolon= \begin{bmatrix}
			\cZd & 0\\
			0 & \hZd
		\end{bmatrix}.
	\end{align*}	
	We now have to check the different items from \Cref{hyp:2} for the shifted hybrid numerical-experimental model~\eqref{eq:hybridSystem:shifted}. We notice that $\check{\mu} = 0 = \hat{\mu}$ implies $\check{\calD}_\mu = \check{F}$ and $\hat{\calD}_\mu = \hat{F}$ and observe 
	\begin{align*}
		\rank\left(\frac{\partial F}{\partial \dot{\state}} \right) = \rank\left(\begin{bmatrix}\
			\frac{\partial \check{F}}{\partial \state_1}  & 0\\
			0 & \frac{\partial \hat{F}}{\partial \state_2}
		\end{bmatrix}\right) = \rank\left(\frac{\partial \check{F}}{\partial \state_1}\right) + \rank\left(\frac{\partial \hat{F}}{\partial \state_2}\right) = \check{a} + \hat{a} = a.
	\end{align*}
	We immediately conclude 
	\begin{multline*}
		\left(\Za^T\frac{\partial F}{\partial \dot{\state}}\right) (t,\state,\dot{\state},\shift{\delay}{\state}) =\\
		\begin{bmatrix}
			\left(\cZa^T\frac{\partial \check{F}}{\partial \dot{\state}_1}\right)\left(t,\state_1,\dot{\state}_1,\shift{\delay}{\hat{G}(t,\state_2)}\right) & 0\\
			0 &  \left(\hZa^T\frac{\partial \hat{F}}{\partial \dot{\state}_2}\right)\left(t,\state_2,\dot{\state}_2,\check{G}(t,\state_1)\right)
		\end{bmatrix} = 0
	\end{multline*}
	such that the first item from \Cref{hyp:2} is satisfied. For the second item we obtain (omitting arguments)
	\begin{align*}
		\hat{a} = \rank\left(\hZa^T \frac{\partial \hat{F}}{\partial \state_2}\right) \leq \rank\left(\begin{bmatrix}
			\cZa^T \frac{\partial\hat{F}}{\partial u_2}\frac{\partial \check{G}}{\partial \state_1} & \hZa^T \frac{\partial \hat{F}}{\partial \state_2}
		\end{bmatrix}\right) \leq \hat{a}
	\end{align*}
	and thus
	\begin{align*}
		\rank \left(\Za^T \frac{\partial F}{\partial \state} \right) &= \rank\left(\begin{bmatrix}
			\cZa^T \frac{\partial \check{F}}{\partial \state_1}& 0\\
			\cZa^T \frac{\partial\hat{F}}{\partial u_2}\frac{\partial \check{G}}{\partial \state_1} & \hZa^T \frac{\partial \hat{F}}{\partial \state_2}
		\end{bmatrix}\right)\\
		&= \rank\left(\cZa^T \frac{\partial \check{F}}{\partial \state_1}\right) + \rank\left(\hZa^T \frac{\partial \hat{F}}{\partial \state_2}\right)
		= \check{a} + \hat{a} = a.
	\end{align*}
	We conclude
	\begin{displaymath}
		\Za^T\frac{\partial F}{\partial z}\Ta = \begin{bmatrix}
			\cZa^T \frac{\partial \check{F}}{\partial z_1}\cTa & 0\\
			\cZa^T \frac{\partial \hat{F}}{\partial u_2}\frac{\partial \check{G}}{\partial \state_1}\hTa - \hZa^T \frac{\partial \hat{F}}{\partial z_2}\hTa^\star\left(\hZa^T \frac{\partial \hat{F}}{\partial \state_2}\hTa^\star \right)^{-1} \cZa^T\frac{\partial \hat{F}}{\partial u_2}\frac{\partial \check{G}}{\partial \state_1}\hTa & \hZa^T \frac{\partial \hat{F}}{\partial z_2}\hTa
		\end{bmatrix} = 0.
	\end{displaymath}
	Similarly as before we have
	\begin{displaymath}
		\rank\left(\frac{\partial F}{\partial \dot{\state}}\Ta\right) = \rank\left(\frac{\partial \check{F}}{\partial \dot{\state}_1}\cTa\right) + \rank\left(\frac{\partial \hat{F}}{\partial \dot{\state}_2}\hTa\right) = \check{d} + \hat{d} = d.
	\end{displaymath}
	The proof follows from
	\begin{displaymath}
		\rank\left(\Zd^T\frac{\partial F}{\partial \dot{\state}}\Ta\right) = \rank\left(\cZd^T \frac{\partial \check{F}}{\partial \dot{\state}_1} \cTa\right) + \rank\left(\hZd^T \frac{\partial \hat{F}}{\partial \dot{\state}_2} \hTa\right) = d.\qedhere
	\end{displaymath}
\end{proof}

\begin{remark}
	The assumption $\delay>0$ is crucial in \Cref{thm:regularStrangenessFree}. In the case $\delay=0$, we have already seen in \Cref{ex:coupledSfreeSystems} that even if both subsystems are strangeness-free, the coupled system might have strangeness-index $\mu>0$.
\end{remark}

In the case that either of the subsystems is not strangeness-free we can proceed as follows. Let 
\begin{align*}
	0 &= \check{D}(t,\state_1,\dot{\state}_1,u_1), & 
	0 &= \hat{D}(t,\state_2,\dot{\state}_2,u_2), &\\
	0 &= \check{A}\left(t,\state_1,u_1,\dot{u}_1,\ldots,u_1^{(\check{\mu})}\right), &
	0 &= \hat{A}\left(t,\state_2,u_2,\dot{u}_2,\ldots,u_2^{(\check{\mu})}\right),\\
	\dot{\state}_1 &= \check{\mathfrak{f}}\left(t,\state_1,u_1,\dot{u}_1,\ldots,u^{(\check{\mu}+1)}\right), &
	\dot{\state}_2 &= \hat{\mathfrak{f}}\left(t,\state_2,u_2,\dot{u}_2,\ldots,u^{(\hat{\mu}+2)}\right)
\end{align*}
denote the strangeness-free reformulations and the underlying ODEs for \eqref{eq:DAE:numerical} and \eqref{eq:DAE:numerical}, respectively. Recall the coupling conditions
\begin{displaymath}
	u_1 = \shift{\delay}{\hat{G}(t,\state_2)}\qquad\text{and}\qquad u_2 = \check{G}(t,\state_1),
\end{displaymath}
which we have to differentiate $\hat{\mu}+1$, respectively $\check{\mu}+1$ times. We observe that in the interval $[0,\delay)$ the coupling condition for $u_1$ does not depend on $\state_2$ but on the history $\history_2$. In particular, we obtain (assuming that $\hat{G}$ is sufficiently smooth)
\begin{align*}
	\dot{u}_1 &= \shift{\delay}{\frac{\partial \hat{G}}{\partial t}(t,\history_2) + \frac{\partial \hat{G}}{\partial \state_2}(t,\history_2)\dot{\history}_2},\\
	\ddot{u}_1 &= \shift{\delay}{\frac{\partial^2 \hat{G}}{\partial t^2}(t,\history) + 2\frac{\partial^2\hat{G}}{\partial t \partial \state_2}(t,\history_2)\dot{\history}_2 + \frac{\partial^2 \hat{G}}{\partial \state_2^2}(t,\history_2)\dot{\history}_2 + \frac{\partial \hat{G}}{\partial\state_2}(t,\history_2)\ddot{\history}_2},
\end{align*}
and similarly for higher derivatives. In particular, there exists functions $\dcheck{D}$, $\dcheck{A}$, and $\dcheck{\mathfrak{f}}$ such that
\begin{align*}
	0 &= \dcheck{D}(t,\state_1,\dot{\state}_1,\shift{\delay}{\history_2}),\\
	0 &= \dcheck{A}\left(t,\state_1,\shift{\delay}{\history_2},\shift{\delay}{\dot{\history}_2},\ldots,\shift{\delay}{\history_2^{(\check{\mu})}}\right),\\
	\dot{\state}_1 &= \dcheck{\mathfrak{f}}\left(t,\state_1,\shift{\delay}{\history_2},\shift{\delay}{\dot{\history}_2},\ldots,\shift{\delay}{\history_2^{(\check{\mu}+1)}}\right)
\end{align*}
for $t\in[0,\delay)$. As an immediate consequence, we can (locally) solve for $\state_1$, provided that the initial trajectory $\history_2$ is sufficiently smooth and the consistency condition
\begin{displaymath}
	0 = \dcheck{A}\left(0,\state_1(0),\history_2(-\delay),\dot{\history}_2(-\delay),\ldots,\history_2^{(\check{\mu})}(-\delay)\right)
\end{displaymath}
is satisfied. On the other hand, the input relation for $u_2$ implies
\begin{align*}
	\dot{u}_2 &= \frac{\partial \hat{G}}{\partial t}(t,\state_1) + \frac{\partial \hat{G}}{\partial \state_1}(t,\state_1)\dot{\state}_1\\
	&= \frac{\partial \hat{G}}{\partial t}(t,\state_1) + \frac{\partial \hat{G}}{\partial \state_1}(t,\state_1)\dcheck{\mathfrak{f}}\left(t,\state_1,\shift{\delay}{\history_2},\shift{\delay}{\dot{\history}_2},\ldots,\shift{\delay}{\history_2^{(\check{\mu}+1)}}\right).
\end{align*}
Note that although derivatives of $\history_2$ up to order $\check{\mu}+1$ appear, $\dot{u}_2$ not necessarily depends on all of them (see for instance\Cref{ex:index:shiftedHybridSys}). In any case, there exists functions $\dhat{D}$, $\dhat{A}$, and~$\dhat{\mathfrak{f}}$ such that
\begin{align*}
	0 &= \dhat{D}(t,\state_2,\dot{\state}_2,\state_1),\\
	0 &= \dhat{A}\left(t,\state_2,\state_1,\shift{\delay}{\history_2},\shift{\delay}{\dot{\history}_2},\ldots,\shift{\delay}{\history_2^{(\check{\mu}+\hat{\mu})}}\right),\\
	\dot{\state}_2 &= \dhat{\mathfrak{f}}\left(t,\state_1,\state_2,\shift{\delay}{\history_2},\shift{\delay}{\dot{\history}_2},\ldots,\shift{\delay}{\history_2^{(\check{\mu}+\hat{\mu} + 1)}}\right).
\end{align*}
Thus, the underlying delay differential equation for the shifted hybrid system~\eqref{eq:hybridSystem:shifted} in $[0,\delay)$ is given by
\begin{equation}
	\label{eq:hybridSystem:shifted:uODE}
	\begin{bmatrix}
		\dot{\state}_1\\
		\dot{\state}_2
	\end{bmatrix} = \begin{bmatrix}
		\dcheck{\mathfrak{f}}\left(t,\state_1,\shift{\delay}{\history_2},\shift{\delay}{\dot{\history}_2},\ldots,\shift{\delay}{\history_2^{(\check{\mu}+1)}}\right)\\
		 \dhat{\mathfrak{f}}\left(t,\state_1,\state_2,\shift{\delay}{\history_2},\shift{\delay}{\dot{\history}_2},\ldots,\shift{\delay}{\history_2^{(\check{\mu}+\hat{\mu} + 1)}}\right)
	\end{bmatrix}
\end{equation}
and the differentiation index is at most $\check{\mu}+\hat{\mu}+1$ and we have shown the following result.

\begin{theorem}
	\label{thm:hybrid:shifted:regular}
	Assume that the subsystems \eqref{eq:DAE:numerical} and \eqref{eq:DAE:experimental} satisfy \Cref{hyp:2} with strangeness index $\check{\mu}$, $\hat{\mu}$, respectively. Then the shifted hybrid numerical-experimental system \eqref{eq:hybridSystem:shifted} has a well-defined differentiation index, which is at most $\check{\mu} + \hat{\mu} + 1$.
\end{theorem}

\begin{example}
	\label{ex:PMSD:hybrid:shifted}
	Shifting the equations for the pendulum in \eqref{eq:PMSD:hybrid} and introducing new variables for the velocities ($v_1 \vcentcolon= \dot{y}_1$, $v_2 \vcentcolon=\dot{x}_2$, and $v_3\vcentcolon=\dot{y}_2$), yields the system
	\begin{subequations}
		\label{eq:PMSD:hybrid:shifted}
		\begin{align}
			\label{eq:PMSD:hybrid:shifted:1}\dot{y}_1 &= v_1,\\
			\label{eq:PMSD:hybrid:shifted:2}\dot{x}_2 &= v_2,\\
			\label{eq:PMSD:hybrid:shifted:3}\dot{y}_2 &= v_3,\\
			\label{eq:PMSD:hybrid:shifted:4}M\dot{v}_1 + Cv_1 + Ky_1 &= f(\shift{\delay}{y_1},\shift{\delay}{y_2},\shift{\delay}{\lambda}),\\
			\label{eq:PMSD:hybrid:shifted:5}m\dot{v}_2 &= -2\lambda x_2,\\
			\label{eq:PMSD:hybrid:shifted:6}m\dot{v}_3 &= -2\lambda(y_2-y_1) - m\mathsf{g},\\
			\label{eq:PMSD:hybrid:shifted:7}0 &= x_2^2 + (y_2-y_1)^2 - L^2,
		\end{align}
	\end{subequations}	
	which is a multibody system with forcing term $f(y_1,y_2,\lambda) = -2\lambda(y_2-y_1)-m\mathsf{g}$ that solely depends on delayed variables. Since multibody systems are special instances of Hessenberg systems, we conclude from \cite[Sec.~4.2]{KunM06} that the shifted hybrid pendulum-mass-spring-damper system has strangeness index $\mu=2$ and satisfies \Cref{hyp:2} with $a = 3$ and $d=4$. The algebraic equations (more precisely the difference equations) are given by
	\begin{align*}
			0 &= x_2^2 + (y_2-y_1)^2 - L^2,\\
			0 &= 2x_2v_2 + 2(y_2-y_1)(v_3-v_1),\\
			0 &= 2v_2^2 + 2(v_2-v_1)^2 - \frac{4}{m}\lambda(x_2^2 + (y_2-y_1)^2)\\
			&\phantom{=} - 2(y_2-y_1)\left(g + \frac{f(\shift{\delay}{y_1},\shift{\delay}{y_2},\shift{\delay}{\lambda})}{M} - \frac{C}{M}v_1 - \frac{K}{M}y_1\right).
	\end{align*}	
	Let us emphasize that despite the higher index, the algebraic equations do not depend on derivatives of $\shift{\delay}{\state}$. Note that also the Lagrange-multiplier is shifted in \eqref{eq:PMSD:hybrid:shifted:4} such that this example is not included in the specific retarded Hessenberg forms as studied in \cite{AscP95}. 
\end{example}

\begin{remark}
	Models with a similar structure as in \eqref{eq:hybridSystem:shifted} and \eqref{eq:hybridSystem:linear:shifted} arise also in the time-discretization via waveform relazation \cite{Ebe08,Mie89,BY11} or the analysis of semi-explicit time-integrators \cite{AltMU19}. 
\end{remark}

\section{Solvability of the hybrid model}
In the previous section we have established that the shifted hybrid numerical-experimental system~\eqref{eq:hybridSystem:shifted} can be solved in the interval $[0,\delay)$ and is regular in the sense of \Cref{thm:wellPosednessDAE}, provided that the subsystems satisfy \Cref{hyp:2} and the history function is admissible. The question that remains to be answered, is whether a solution exists on time intervals $[0,T)$ with $T>\delay$. 

\begin{remark}
	For the linear time-invariant case, this is discussed in detail in \cite{TreU19} for a distributional solution concept and in \cite{Ung18,HaM12,Cam80,Cam91,Cam95,Pop06} for the solution concept as defined in \Cref{def:solutionConcept}. Results for linear time-varying systems are developed in \cite{HaMS14,HaM16}. Moreover, a special class of nonlinear DDAEs is discussed in \cite{AscP95}.
\end{remark}

In view of the method of steps discussed in the previous section, the question that remains to be answered is, which conditions on the subsystems and the history function ensure that the initial condition \eqref{eq:initialConditionDAE} is consistent for all $i=1,\ldots,\maxMOS$. Unfortunately, regularity of the subsystems and an admissible history function are not sufficient to establish a solution for $t>\delay$, see for instance \cite{Ung18}.
\begin{example}
	\label{ex:advancedDDAE}
	Consider the regular DDAE
	\begin{displaymath}
			\dot{x}(t) = y(t), \qquad 0 = x(t) - y(t-1).
	\end{displaymath}
	Applying the method of steps \Cref{eq:sequenceDAEs} yields
	\begin{align}
		\mos{x}{i} = \mos{y}{i-1}\qquad\text{and}\qquad \mos{y}{i} = \dmos{y}{i-1}.
	\end{align}
	With the history function $\history(t) = \begin{smallbmatrix}0\\t+1\end{smallbmatrix}$ we obtain $\mos{x}{1}(t) = t$ and $\mos{y}{1}(t) = 1$, and we deduce that the history function is admissible. However, the initial value $\mos{y}{1}(1) = 1$ is not consistent for the associated DAE on the interval $[1,2)$. In particular, the solution exists only on the interval $[0,1)$.
\end{example}
The issue in the previous example is that the equation $z_i = \dot{z}_{i-1}$ results in solutions $z_i$ that become less smooth for increasing $i$ and possible discontinuities of the form
\begin{displaymath}
	\mos{\state^{(k)}}{i-1}(\delay^-) \neq \mos{\state^{(k)}}{i}(0)
\end{displaymath}
are propagated to $\mos{\state^{(k-1)}}{i}(\delay^-) \neq \mos{\state^{(k-1)}}{i+1}(0)$. The discontinuity propagation leads to the following classification \cite{BelC63}:
The scalar equation
\begin{displaymath}
	a_0\dot{\state}(t) + a_1\dot{\state}(t-\delay) + b_0\state(t) + b_1\state(t-\delay) = f(t)
\end{displaymath}
is of \emph{retarded type} if $a_0\neq0$ and $a_1=0$, of \emph{neutral type} if $a_0\neq0$ and $a_1\neq 0$ and of \emph{advanced type} if $a_0\neq 0$ and $a_1\neq 0$. However, as the previous example indicates, one has to be careful, how this classification can be extended to the vector-valued case (cf.~\cite{Cam91}). Following \cite{HaM16} we make the following definition. 
\begin{defn}
	\label{def:RNA}
	Let 
	\begin{equation}
		\label{eq:uDDE}
		\dot{\state} = \mathfrak{f}\left(t,\state,\shift{\delay}{\state},\ldots,\shift{\delay}{\state^{(s)}}\right)
	\end{equation}
	denote the underlying DDE of the DDAE~\eqref{eq:DDAE} and assume $\tfrac{\partial \mathfrak{f}}{\partial \shift{\delay}{\state^{(s)}}} \not\equiv 0$. Then~\eqref{eq:DDAE} is called \emph{retarded}, \emph{neutral}, or \emph{advanced} if $s=0$, $s=1$, or $s\geq 2$ in~\eqref{eq:uDDE}.
\end{defn}
\begin{lemma}
	\label{lem:notAdvanced}
	Assume that the DDAE~\eqref{eq:DDAE} satisfies \Cref{hyp:2} and let
	\begin{subequations}
		\label{eq:DDAE:sfree}
	\begin{align}
		\label{eq:DDAE:differentialPart}
		0 &= D\left(t,\state,\dot{\state},\shift{\delay}{\state}\right),\\
		\label{eq:DDAE:algebraicPart}		
		0 &= A\left(t,\state,\shift{\delay}{\state},\shift{\delay}{\dot{\state}},\ldots,\shift{\delay}{\state^{(s-1)}}\right)
	\end{align}
	\end{subequations}
	denote the associated strangeness-free reformulation with the convention that either $A$ does not depend on $\shift{\delay}{\state^{(k)}}$ for any $k\in\mathbb{N}$, or
	\begin{displaymath}
		\frac{\partial A}{\partial \shift{\delay}{\state^{(s-1)}}} \not\equiv 0
	\end{displaymath}
	Then~\eqref{eq:DDAE} is retarded, neutral, or advanced, if $\tfrac{\partial A}{\partial \shift{\delay}{\state^{(k)}}} \equiv 0$ for all $k\in\mathbb{N}$, $s=1$, or $s=2$, respectively.
\end{lemma}

\begin{proof}
	The proof follows immediately from rewriting~\eqref{eq:DDAE:sfree} as in~\eqref{eq:sFreeExplicit:control} and \eqref{eq:uODE:control}.
\end{proof}

\begin{theorem}
	\label{thm:solvabilityDDAE}
	Suppose that the DDAE~\eqref{eq:DDAE} is sufficiently smooth, has strangeness-index $\mu$, satisfies \Cref{hyp:2} with characteristic values $\mu,a,d$ and $\mu+1,a,d$, is not advanced, and the history function $\history_0\in\mathcal{C}^1([0,\tau],\mathbb{R}^{\stateDim})$ is admissible. Then the initial trajectory problem~\eqref{eq:DDAE:IVP} is solvable.
\end{theorem}

\begin{proof}
	Since the DDAE~\eqref{eq:DDAE} satisfies \Cref{hyp:2} and is not advanced, \Cref{lem:notAdvanced} implies that the strangeness-free reformulation is of the form
	\begin{align}
		\label{eq:DDAE:sfree:notAdvanced}
		0 &= D(t,\state,\dot{\state},\shift{\delay}{\state}), & 
		0 &= A(t,\state,\shift{\delay}{\state})
	\end{align}
	with the understanding that $A$ may not depend on $\shift{\delay}{\state}$. Applying the method of steps to \eqref{eq:DDAE:sfree:notAdvanced} yields the sequence of initial value problems
	\begin{equation}
		\label{eq:DDAE:sfree:MOS}
		\begin{aligned}
			0 &= D\left(t + (i-1)\delay,\mos{\state}{i},\dmos{\state}{i},\mos{\state}{i-1}\right), \\
			0 &= A\left(t+ (i-1)\delay,\mos{\state}{i},\mos{\state}{i-1}\right),\\
			\mos{\state}{i}(0) &= \mos{\state}{i-1}(\delay^-).
		\end{aligned}
	\end{equation}		
	Since the history function is admissible, we can (locally) solve \eqref{eq:DDAE:sfree:MOS} for $i=1$ and by \cite[Theorem~4.13]{KunM06} this solution also is a solution of~\eqref{eq:sequenceDAEs}. Although this solution is of local nature it can be globalized by applying the cited theorem again until we reach the boundary of $\calM_{\mu}$ (cf. \cite[Remark~4.14]{KunM06}). If we assume that the solution exists on the time interval $[0,\tau)$ this immediately implies 
\begin{displaymath}
	0 = \lim_{t\nearrow \delay} A(t,\mos{\state}{1}(t),\mos{\state}{0}(t)) = A(\delay,\mos{\state}{1}(\delay^-),\mos{\state}{0}(\delay)).
\end{displaymath}
Hence, $\mos{\state}{1}(\delay^-)$ is consistent for the DAE~\eqref{eq:sequenceDAEs} with $i=2$. The result follows iteratively by repeating this procedure.
\end{proof}

\begin{corollary}
	\label{thm:solvability:hybrid}
	Suppose that the numerical and experimental subsystem both satisfy \Cref{hyp:2} with $\mu=0$. Then for any $\delay>0$ and for any admissible history function $\history$, the initial trajectory problem for the shifted hybrid numerical-experimental system~\eqref{eq:hybridSystem:shifted} is solvable.
\end{corollary}

\begin{proof}
	\Cref{thm:regularStrangenessFree} ensures that the shifted hybrid system is strangeness-free. From \Cref{lem:notAdvanced} we deduce that \eqref{eq:hybridSystem:shifted} is not advanced. The result thus follows immediately from \ref{thm:solvabilityDDAE}.
\end{proof}

\begin{example}
	Although the system for the pendulum~\eqref{eq:pendulum} is not strangeness-free, \Cref{ex:PMSD:hybrid:shifted} shows that the shifted hybrid system resulting from coupling the pendulum with the mass-spring-damper system is not advanced. In particular, \Cref{thm:solvabilityDDAE} ensures that the associated initial trajectory problem is solvable.
\end{example}

\begin{remark}
	Even if the DDAE~\eqref{eq:DDAE} is advanced, the initial trajectory problem~\eqref{eq:DDAE:IVP} may have a solution if the history function satisfies additional \emph{splicing conditions} \cite{BelZ03,Ung18}. In the linear case, this is analyzed in \cite{Ung18} for the case that the index of the matrix pencil $(E,A)$ is less than or equal three. We conjecture that a similar analysis is possible for the nonlinear case as well. 
\end{remark}
%
\section{Conclusions}
In this paper, we have analyzed the solvability of (nonlinear) delay differential-algebraic equations (DDAEs) that arise in real-time dynamic substructuring. We show in a sequence of results (\Cref{thm:regularStrangenessFree}, \Cref{thm:solvabilityDDAE}, and \Cref{thm:solvability:hybrid}) that if both subsystems are strangeness-free, then the initial trajectory problem for the hybrid model is solvable. The motivational example of a coupled pendulum-mass-spring-damper system, described in \Cref{subsec:motivationalExample}, that we solvability can also be obtained for systems that are not strangeness-free.

The analysis indicates the need for an energy-based formulation of the subsystems (cf.~\Cref{rem:pH}) to reduce the difficulties associated with higher index problems.  We expect that the specific structure of so-called port-Hamiltonian descriptor systems \cite{BeaMXZ18}, a particular energy-based modeling paradigm, simplifies the analysis of the hybrid numerical-experimental setup and provides more robust models with respect to perturbations and uncertainty. Consequently, an extension of port-Hamiltonian systems to delay systems is an interesting open research question. Another further research direction is to extend the ideas from \cite{Ung18} that allow rewriting a specific class of advanced systems as non-advanced systems, provided that the history function satisfies so-called splicing conditions \cite{BelZ03}.
%
\bibliographystyle{plain}
\bibliography{references}

\end{document}